\documentclass[a4paper,pdftex,reqno,10pt]{amsart}

\allowdisplaybreaks[1]

\usepackage{geometry}

\usepackage{graphicx}
\usepackage{enumerate}
\usepackage{courier}
\usepackage{times}
\usepackage{amsmath,amssymb}
\usepackage{algorithm}
\usepackage{algorithmic}
\usepackage{hyperref}
\usepackage{mathtools}
\usepackage{color}

\theoremstyle{plain}
\newtheorem{Theorem}{Theorem}[section]
\newtheorem{Proposition}[Theorem]{Proposition}

\newtheorem{Corollary}[Theorem]{Corollary}
\newtheorem{Lemma}[Theorem]{Lemma}

\theoremstyle{definition}
\newtheorem{Definition}[Theorem]{Definition}

\newtheorem{Example}[Theorem]{Example}

\theoremstyle{remark}

\DeclareMathOperator{\rk}{rk}
\newcommand{\gaussm}[3]{\genfrac{[}{]}{0pt}{}{#1}{#2}_{#3}}
\newcommand{\gauss}[3]{\genfrac{[}{]}{0pt}{}{\mathbb{F}_{#3}^{#1}}{#2}}

\newcommand{\rdist}{\mathrm{d}_{\mathrm{r}}}
\newcommand{\fdist}{\mathrm{d}_{\mathrm{G}}}
\newcommand{\sdist}{\mathrm{d}_{\mathrm{s}}}
\newcommand{\idist}{\mathrm{d}_{\mathrm{i}}}

\newcommand{\MRD}{\texttt{MRD}} 
\newcommand{\highlight}[1]{#1}

\begin{document}


\title{Bounds for flag codes}

\date{}

\author{Sascha Kurz}
\address{Sascha Kurz, University of Bayreuth, 95440 Bayreuth, Germany}
\email{sascha.kurz@uni-bayreuth.de}

\begin{abstract}
The application of flags to network coding has been introduced recently, see e.g.\ \cite{with_flags}. It is a variant 
to random linear network coding and explicit routing solutions for given networks. Here we study lower and upper bounds for the 
maximum possible cardinality of a corresponding flag code with given parameters.

\end{abstract}

\maketitle

\section{Introduction}
\label{sec_intro}

Let $q$ be a prime power and $\mathbb{F}_q$ the finite field with $q$ elements. For given integers 
$1\le k\le v$ a $k$-dimensional subspace $U$ of $\mathbb{F}_q^v$ is called a \emph{$k$-space} (in $\mathbb{F}_q^v$). 
Sometimes we also use the language of projective geometry, i.e., we \highlight{speak} of \emph{points}, \emph{lines}, \emph{planes}, 
and \emph{hyperplanes} for $1$-spaces, $2$-spaces, $3$-spaces, 
and $(v-1)$-spaces, respectively.  
The set of all $k$-spaces in $\mathbb{F}_q^v$ is abbreviated by $\gauss{v}{k}{q}$ and its cardinality 
is denoted by the \emph{$q$-binomial Gaussian coefficient} $\gaussm{v}{k}{q}=\prod_{i=1}^{k} \frac{q^{v-k+i}-1}{q^{i}-1}$. 
A \emph{full flag} over $\mathbb{F}_q^v$ is a sequence of nested subspaces with dimensions from $1$ to $v-1$. 
If not all of these dimensions need to occur, we speak of a \emph{flag}. (Full) \emph{flag codes} are collections 
of flags. The use of flag codes for network coding was proposed in \cite{with_flags}. In \cite{phd_liebhold} the author argues that 
subspace coding with flags can be ranged between random linear network coding, using constant dimension codes, and optimized 
routing solutions, whose computation is time-consuming. For special multicast networks network coding solutions also lead to hard 
combinatorial problems, see e.g.~\cite{cai2019network, etzion2020subspace} for so-called generalized combination networks. 
Here, we will not go into the details of the used chanel model or comparisons 
with other methods for network coding. Moreover, we will not consider the problem of encoding and decoding algorithms. The interested reader 
can find more details on this e.g.\ in \cite{fourier2020degenerate,phd_liebhold,with_flags,liebhold2018generalizing}. Here we study 
lower and upper bounds for the maximum possible cardinality $A_q^f(v,d)$ of those flag codes.

The remaining part of this paper is organized as follows. In Section~\ref{sec_preliminaries} we introduce the necessary basic definitions 
and the first bounds for $A_q^f(v,d)$. An integer linear programming formulation for the exact determination of $A_q^f(v,d)$ is the topic 
of Section~\ref{sec_ILP}. Parametric bounds on the maximum possible codes sizes are determined in Section~\ref{sec_bounds}. The case of non-full 
\highlight{flags} and other variants are broached in Section~\ref{sec_non_full}. We summarize the obtained exact values and bounds for $A_q^f(v,d)$ for small 
parameters in Section~\ref{sec_exact_values_and_bounds}. The paper is finished with a brief conclusion and a few remarks on open problems and 
future research directions in Section~\ref{sec_conclusion}.  

\section{Preliminaries and first bounds}
\label{sec_preliminaries}

In the following $q$ is always a prime power. For two subspaces $U,W$ in $\mathbb{F}_q^v$ we write $U\le W$ iff $U$ is 
contained in $W$. If $U\le W$ and $U\neq W$, then we write $U<W$. The dimension of a subspace $U$ of $\mathbb{F}_q^v$ 
is denoted by $\dim(U)$. The set of all subspaces of $\mathbb{F}_q^v$ is turned into a metric space via the \emph{injection 
distance}
$$
  \idist(U,W)=\dim(U+W)-\min\{\dim(U),\dim(W)\}=\max\{\dim(U),\dim(W)\}-\dim(U\cap W)
$$  
or the \emph{subspace distance}
$$
  \sdist(U,W)=\dim(U+W)-\dim(U\cap W)=\dim(U)+\dim(W)-2\cdot\dim(U\cap W).
$$
Note that for $\highlight{U},W\in\gauss{v}{k}{q}$ we have
\begin{eqnarray*} 
  \idist(U,W) &=& \dim(U+W)-k=k-\dim(U\cap W)\,\,\text{ and}\\
  \sdist(U,W) &=& 2k-2\dim(U\cap W)=2\cdot\idist(U,W).
\end{eqnarray*} 
By $A_q^i(v,d;k)$ we denote the maximum possible cardinality of a set $\mathcal{C}\subseteq\gauss{v}{k}{q}$, where $\idist(U,W)\ge d$ 
for all pairs of different elements $U$, $W$ of $\mathcal{C}$. Replacing the injection distance by the subspace distance we obtain 
$A_q^s(v,d;k)$, where $A_q^i(v,d;k)=A_q^s(v,2d;k)$. Bounds for $A_q^s(v,2d;k)$ can be found in \cite{heinlein2016tables} and the 
corresponding online tables at \url{www.subspacecodes.uni-bayreuth.de}.  
 
\begin{Lemma}
  \label{lemma_intersection_distance}
  For two subspaces $U,W\in\gauss{v}{k}{q}$ the following statements are equivalent
  \begin{itemize}
    \item[(1)] $\idist(U,W)\le d$;
    \item[(2)] $\dim(U\cap W)\ge k-d$;
    \item[(3)] $\dim(U+W)\le k+d$;
    \item[(4)] there exists a subspace $X\le\mathbb{F}_q^v$ with $X\le U$, $X\le W$, and $\dim(X)\ge k-d$; and 
    \item[(5)] there exists a subspace $X\le\mathbb{F}_q^v$ with $X\ge U$, $X\ge W$, and $\dim(X)\le k+d$;
  \end{itemize}
\end{Lemma} 
\begin{proof}
  The equivalence of (1)-(3) is obvious from the definition. For (4) we remark that the conditions $X\le U$ and $X\le W$ are equivalent to 
  $X\le U\cap W$. Similarly, for (5) the conditions $X\ge U$ and $X\ge W$ are equivalent to 
  $X\ge U+W$.  
\end{proof} 
 
\begin{Definition} 
   A \emph{flag} is a list of subspaces $\Lambda=\left(W_1,\dots,W_m\right)$ of $\mathbb{F}_q^v$ with 
   $$
     \{0\}<W_1<\dots<W_m<\mathbb{F}_q^v.
   $$
   The \emph{type} of $\Lambda=\left(W_1,\dots,W_m\right)$ is the set of dimensions
  $$
    \operatorname{type}(\Lambda):=\left\{\dim(W_i)\mid 1\le i\le m\right\}\subseteq\left\{1,\dots ,\highlight{v-1}\right\}.
  $$
  Let 
  $$
    \mathcal{F}(v,q):=\left\{\Lambda\mid \Lambda\text{ is a flag in }\mathbb{F}_q^v\right\}
  $$
  denote the set of all flags in $\mathbb{F}_q^v$ and for $T\subseteq\{1,\dots,v-1\}$ let 
  $$
    \mathcal{F}_T(v,q):=\left\{\Lambda\in\mathcal{F}(v,q)\mid \operatorname{tpye}(\Lambda)=T\right\}
  $$
  be the set of all flags of $\mathbb{F}_q^v$ of type $T$\highlight{.}
\end{Definition} 

As noted in \cite{with_flags}, the intersection of two flags is again a flag \highlight{and the} set of all flags in $\mathbb{F}_q^v$ forms 
a simplicial complex (with respect to inclusion). There the authors give all relevant facts about the spherical building of the general 
linear group of a finite dimensional vector space. Here we will not use the language of buildings. If a flag in $\mathbb{F}_q^v$ has type $\{1,\dots,v-1\}$, then we speak of a 
\emph{full flag} whose set is \highlight{denoted} by $\mathcal{F}_f(q)$. Full flags are the maximal simplices while the unique minimal flag is the empty 
set with type $\emptyset$. The second minimal flags $\{W\}$ are the proper subspaces $W$ of $\mathbb{F}_q$. So, the Grassmannian of all $k$-dimensional 
subspaces, i.e., $\gauss{v}{k}{q}$, is in bijection with the set of flags $\mathcal{F}_{\{k\}}(q)$ of type $\{k\}$. 

\begin{Definition}
  Let $\Lambda=\left(W_1,\dots,W_m\right)$ and $\Lambda'\highlight{=}\left(W_1',\dots,W_m'\right)$ be two flags of $\mathbb{F}_q^v$ of the 
  same type $T=\left\{k_1,\dots,k_m\right\}$ with $k_i=\dim(W_i)=\dim(W_i')$ for all $1\le i\le m$. Then, the \emph{Grassmann distance} is defined as
  $$
    \fdist(\Lambda,\Lambda'):=\sum_{i=1}^m \idist(W_i,W_i')=\sum_{i=1}^m \left(k_i-\dim(W_i\cap W_i')\right).
  $$
\end{Definition}

So, for $m=1$ the Grassmann distance corresponds to the injection distance, i.e., half the subspace distance, between $W_1$ and $W_1'$. 
For $U,W\in \gauss{v}{k}{q}$ we have $0\le \idist(U,W)\le \min\{k,v-k\}$, so that we set 
$$
  m(v,T)=\left(\min\{k_1,v-k_1\},\dots,\min\{k_m,v-k_m\}\right),
$$ 
where $T=\left\{k_1,\dots,k_m\right\}\subseteq\{1,\dots,v-1\}$ with $k_1<\dots<k_m$. If $T=\{1,\dots,v-1\}$ we just write $m(v)$ instead of 
$m(v,T)$. By $x_i$ we denote the $i$th component for each vector $x\in\mathbb{R}^n$. With this we can state
$$
  \fdist(\Lambda,\Lambda')\le \sum_i m(v,T)_i
$$ 
for all $\Lambda,\Lambda'\in\mathcal{F}_T(v,q)$. As mentioned in 
\cite[Remark 4.5]{with_flags} we have $1\le \fdist(\Lambda,\Lambda')\le \left\lfloor (v/2)^2\right\rfloor$ for two distinct flags in $\mathbb{F}_q^v$. A 
\emph{flag code} $\mathcal{C}$ of type $T$ is a collection of flags in $\mathbb{F}_q^v$ of type $T$. If $\#\mathcal{C}\ge 2$, then the minimum distance 
$\fdist(\mathcal{C})$ is the minimum of $\fdist(\Lambda,\Lambda')$ over all pairs of distinct elements $\Lambda,\Lambda'\in\mathcal{C}$. For $\#\mathcal{C}<2$ 
we set $\fdist(\mathcal{C})=\infty$. By $A_q^f(v,d;T)$ we denote the maximum possible cardinality of a flag code $\mathcal{C}$ of type $T$ in $\mathbb{F}_q^v$ 
that has minimum distance at least $d$. The case of full flags, i.e.\ $T=\{1,\dots,v-1\}$, is abbreviated as $A_q^f(v,d)$. Technically, we set 
$A_q^f(v,d)=1$ if $d>\left\lfloor(v/2)^2\right\rfloor$ and restrict ourselves to $1\le d\le \left\lfloor(v/2)^2\right\rfloor$ in the following. The \emph{dual} 
of a flag $\Lambda=\left(W_1,\dots,W_m\right)$ in $\mathbb{F}_q^v$ of type $T\subseteq \{1,\dots,v-1\}$, denoted by $\Lambda^\top$, is given by 
$\left(W_{\highlight{m}}^\top,\dots,W_{\highlight{1}}^\top\right)$. Since we have $\idist(U,W)=\idist\!\left(U^\top,W^\top\right)$ for each $U,W\in\gauss{v}{k}{q}$, for some arbitrary 
integer $k$, the minimum Grassmann distance $d(\mathcal{C})$ of a flag code of type $T$ in $\mathbb{F}_q^v$ is the same as $d\!\left(\mathcal{C}^\top\right)$, 
where $\mathcal{C}^\top:=\left\{\Lambda^\top\mid\Lambda\in\mathcal{C}\right\}$. Moreover, we have
$$
  \operatorname{type}\!\left(\mathcal{C}^\top\right)=\left\{v-t\mid t\in\operatorname{type}(\mathcal{C})\right\}=:T^\top,
$$
so that $A_q^f(v,d;T)=A_q^f\!\left(v,d;T^\top\right)$. The aim of this paper is to derive bounds on $A_q^f(v,d;T)$ and mostly on $A_q^f(v,d)$.

The arguably easiest case for the determination of $A_q^f(v,d;T)$ is minimum distance $d=1$, where $A_q^f(v,\highlight{1};T)=\#\mathcal{F}_T(v,q)$. If $T=\left\{k_1,\dots,k_m\right\}$ 
with $0<k_1<\dots<k_m<v$, then we have
\begin{equation}
  \label{eq_count_flags}
  A_q^f(v,1;T)=\gaussm{v}{k_1}{q}\cdot \prod_{i=2}^m \gaussm{v-k_{i-1}}{k_i-k_{i-1}}{q} 
\end{equation}
and
\begin{equation}
  A_q^f(v,1)=\prod_{i=2}^{v} \frac{q^i-1}{q-1}. 
\end{equation}

For the maximum possible minimum distance $d=\left\lfloor(v/2)^2\right\rfloor$ we have:
\begin{Proposition}
  \label{prop_a_max_distance}
  For each integer $k\ge 1$ we have
  $$
    A_q^f(2k,k^2)=q^k+1  
  $$
  and for each integer $k\ge 2$ we have
  $$
    A_q^f(2k+1,k^2+k)=q^{k+1}+1.  
  $$
\end{Proposition}
\begin{proof}
  Let $\mathcal{C}$ be a full flag code in $\mathbb{F}_q^v$ with the maximum possible minimum distance $d=\left\lfloor(v/2)^2\right\rfloor$, where $v\ge 2$. 
  If $\Lambda=\left(W_1,\dots,W_{v-1}\right)$ and $\Lambda'=\left(W_1',\dots,W_{v-1}'\right)$ are two different elements of $\mathcal{C}$ with 
  $\dim(W_i)=\dim(W_i')=i$ for all $1\le i\le v-1$, then we have 
  $$
    i-\dim(W_i\cap W_i')=\min\{i,v-i\},
  $$ 
  i.e., $W_i$ and $W_i'$ have the maximum possible injection distance $\idist(W_i,W_i')$. So, we clearly have the upper bounds $A_q^f(2k,k^2)\le A_q^i(2k,k;k)=q^k+1$ 
  and $A_q^f(2k+1,k^2+k)\le A_q^i(2k+1,k;k)=q^{k+1}+1$ (using $k\ge 2$), where the maximum possible codes sizes for the injection distance are well known, 
  see e.g.\ \cite{beutelspacher1975partial} or \cite{heinlein2016tables}.
  
  For the construction let $\mathcal{C}_k$ be a set of $k$-spaces in $\mathbb{F}_q^{v}$, where $v=2k$, with minimum injection distance $\idist\!\left(\mathcal{C}_k\right)=k$ and 
  cardinality $A_q^i(2k,k;k)=q^k+1$, i.e., a $k$-spread in $\mathbb{F}_q^{2k}$. We extend each element $W_k\in\mathcal{C}_k$ to a full flag $\left(W_1,\dots,W_{v-1}\right)$ by choosing 
  $W_i\subset\neq W_{i+1}$ with $\dim(W_i)=i$ arbitrarily for $i=k-1,\dots,1$. Similarly, we choose $W_{i}\supsetneq W_{i-1}$ with $\dim(W_i)=i$ arbitrarily for $i=k+1,\dots,v-1$. 
  This gives a full flag code $\mathcal{C}$ in $\mathbb{F}_q^{2k}$ of cardinality $q^k+1$. Now let $\Lambda=\left(W_1,\dots,W_{v-1}\right)$ and $\Lambda'=\left(W_1',\dots,W_{v-1}'\right)$ 
  be two different elements of $\mathcal{C}$ with $\dim(W_i)=\dim(W_i')=i$ for all $1\le i\le v-1$. Since $\dim(W_k\cap W_k')=0$, we have $\dim(W_i\cap W_i')=0$ and 
  $i-\dim(W_i\cap W_i')=\min\{i,2k-i\}$ for all $1\le i\le k$. For $k\le i\le v-1$ we can easily check $\dim(W_i\cap W_i')=i-k$ and $i-\dim(W_i\cap W_i')=\min\{i,2k-i\}$. 
  Thus, $\mathcal{C}$ has the maximum possible Grassmann distance.
  
  For the ambient space $\mathbb{F}_q^{v}$, where $v=2k+1$, let $\mathcal{C}_k$ be a set of $k$-spaces in $\mathbb{F}_q^{2k+1}$ with minimum injection distance $\idist\!\left(\mathcal{C}_k\right)=k$ 
  and cardinality $A_q^i(2k+1,k;k)=q^{k+1}+1$, i.e., a partial $k$-spread of maximum possible size in $\mathbb{F}_q^{2k+1}$. Now let $P$ be a point in $\mathbb{F}_q^{2k+1}$, i.e., a 
  $1$-space, that is not contained in an element of $\mathcal{C}_k$. (Since $\gaussm{k}{1}{q}\cdot \left(q^{k+1}+1\right)<\gaussm{2k+1}{1}{q}$, such a point $P$ exists.) We extend 
  each element $W_k\in\mathcal{C}_k$ to a full flag $\left(W_1,\dots,W_{\highlight{v-1}}\right)$ by choosing $W_i\highlight{\subsetneq} W_{i+1}$ with $\dim(W_i)=i$ arbitrarily for $i=k-1,\dots,1$. 
  The $(k+1)$-space $W_{k+1}$ is defined by $W_{k+1}=\langle W_k,P\rangle$. Similarly as before, we choose $W_{i}\supsetneq W_{i-1}$ with $\dim(W_i)=i$ arbitrarily for $i=k+2,\dots,v-1$. This 
  gives a full flag code $\mathcal{C}$ in $\mathbb{F}_q^{2k+1}$ of cardinality $q^{k+1}+1$. Given two different elements $\Lambda=\left(W_1,\dots,W_{v-1}\right)$ and 
  $\Lambda'=\left(W_1',\dots,W_{v-1}'\right)$ of $\mathcal{C}$ with $\dim(W_i)=\dim(W_i')=i$ for all $1\le i\le v-1$, we can easily check $i-\dim(W_i\cap W_i')=\min\{i,v-i\}$, 
  i.e., $\mathcal{C}$ attains the maximum possible minimum Grassmann distance.     
\end{proof}

We remark that the case $v=2k$ of Proposition~\ref{prop_a_max_distance} was independently proven in \cite{alonso2020flag}, where the authors also give a decoding algorithm and 
further details.

 \begin{Proposition}
  \label{prop_a_3_2}
  $$A_q^f(3,2)=\gaussm{3}{1}{q}=q^2+q+1$$
\end{Proposition} 
\begin{proof}
  Let $\mathcal{C}$ be a full flag code in $\mathbb{F}_q^3$ with minimum Grassmann distance $d=2$. Suppose there are two different elements $\Lambda=\left(W_1,W_2\right)$ 
  and $\Lambda'=\left(W_1',W_2'\right)$ in $\mathcal{C}$ with $W_1=W_1'$. Then, we have $\idist(W_1,W_1')=0$ and $\idist(W_2,W_2')\le 1$, so that 
  $\fdist(\Lambda,\Lambda')\le 1\highlight{<2}$. Thus, we have $\#\mathcal{C}\le \gaussm{3}{1}{q}=q^2+q+1$, \highlight{which is the number of choices for $W_1$}.
  
  For the lower bound we construct a matching code using the Singer group $\langle\sigma\rangle$ generated by a Singer cycle $\sigma$ of $\mathbb{F}_q^3$, i.e., 
  $\langle\sigma\rangle\le \operatorname{P\Gamma L}(3,q)$ is the cyclic group of order $\gaussm{3}{1}{q}=q^2+q+1$ that acts regularly on the set of points or hyperplanes, 
  \highlight{see e.g.\ \cite{drudge2002orbits}}.
  Now let $L$ be an arbitrary line in $\mathbb{F}_q^3$ and $P\le L$ and arbitrary point. With this we set $\Lambda:=\left(P,L\right)$ and 
  $\mathcal{C}=\Lambda^{\langle \sigma\rangle}:=\left\{\Lambda^g\mid g\in \langle \sigma\rangle\right\}$, where $\Lambda^g=\left(P^g,L^g\right)$ and $U^g$ denotes the 
  application of $g\in P\Gamma L(v,q)$ onto a subspace $U$ in $\mathbb{F}_q^v$. For two different group elements $g_1,g_2\in \langle \sigma\rangle$ we have 
  $\idist\!\left(P^{g_1},P^{g_2}\right)=1$ and $\idist\!\left(L^{g_1},L^{g_2}\right)=1$, so that $\fdist(\mathcal{C})=2$.
\end{proof}

\begin{Proposition}
  \label{prop_a_4_3}
  $$A_q^f(4,3)=\gaussm{4}{1}{q}=q^3+q^2+q+1$$
\end{Proposition}
\begin{proof}
  Let $\mathcal{C}$ be a full flag code in $\mathbb{F}_q^4$ with minimum Grassmann distance $d=3$. Suppose there are two different elements $\Lambda=\left(W_1,W_2,W_3\right)$ 
  and $\Lambda'=\left(W_1',W_2',W_3'\right)$ in $\mathcal{C}$ with $W_1=W_1'$. Then, we have $\idist(W_1,W_1')=0$, $\idist(W_2,W_2')\le 1$, and $\idist(W_3,W_3')\le 1$, so that 
  $\fdist(\Lambda,\Lambda')\le 2\highlight{<3}$. Thus, we have $\#\mathcal{C}\le \gaussm{4}{1}{q}=q^3+q^2+q+1$, \highlight{which is the number of choices for $W_1$}.
  
  For the lower bound we construct a matching code using the Singer group $\langle\sigma\rangle$ generated by a Singer cycle $\sigma$ of $\mathbb{F}_q^4$, i.e., 
  $\langle\sigma\rangle\le \operatorname{P\Gamma L}(4,q)$ is the cyclic group of order $\gaussm{4}{1}{q}$ that acts regularly on the set of points or hyperplanes. 
  As shown in \cite{drudge2002orbits}, see also \cite{glynn1988set} for this special case, the action of a Singer group partitions the set of $\gaussm{4}{2}{q}=
  (q^2 + 1)\cdot (q^2 + q + 1)$ lines into orbits of size $q^2+1$ or $q^3+q^2+q+1$. More precisely, there exists exactly one orbit of length $q^2+1$, the geometric 
  line spread, and $q$ orbits of length $q^3+q^2+q+1$. Let $\mathcal{L}$ be an orbit of the latter and $L\in\mathcal{L}$ one of the $q+1$ elements that contain $P$ and 
  $H$ be an arbitrary hyperplane containing $L$. With this we set $\Lambda:=\left(P,L,H\right)$ and $\mathcal{C}=\Lambda^{\langle \sigma\rangle}:=\left\{\Lambda^g\mid g\in 
  \langle \sigma\rangle\right\}$, where $\Lambda^g=\left(P^g,L^g,H^g\right)$ and $U^g$ denotes the application of $g\in P\Gamma L(v,q)$ onto a subspace $U$ in $\mathbb{F}_q^v$. 
  For two different group elements $g_1,g_2\in \langle \sigma\rangle$ we have $\idist\!\left(P^{g_1},P^{g_2}\right)=1$, $\idist\!\left(L^{g_1},L^{g_2}\right)\ge 1$, and 
  $\idist\!\left(H^{g_1},H^{g_2}\right)=1$, so that $\fdist(\mathcal{C})\ge 3$.     
\end{proof}

Exemplarily we state an upper bound on the maximum cardinality of a full flag code for the next open case:
\begin{Proposition}
  \label{prop_a_4_2}
  $$A_q^f(4,2)\le\gaussm{4}{1}{q}\cdot \gaussm{3}{1}{q}=\left(q^3+q^2+q+1\right)\cdot \left(q^2 + q + 1\right)=q^5 + 2q^4 + 3q^3 + 3q^2 + 2q + 1$$
\end{Proposition}
\begin{proof}
  Let $\mathcal{C}$ be a full flag code in $\mathbb{F}_q^4$ with minimum Grassmann distance $d=2$. Suppose there are two different elements $\Lambda=\left(W_1,W_2,W_3\right)$ 
  and $\Lambda'=\left(W_1',W_2',W_3'\right)$ in $\mathcal{C}$ with $W_1=W_1'$ and $W_2=W_2'$. Then, we have $\idist(W_1,W_1')=0$, $\idist(W_2,W_2')=0$, and $\idist(W_3,W_3')\le 1$, 
  so that $\fdist(\Lambda,\Lambda')\le 1\highlight{<2}$. Thus, we have $\#\mathcal{C}\le \gaussm{4}{1}{q}\cdot\gaussm{3}{1}{q}$, \highlight{which is the number of choices for 
  $\left(W_1,W_2\right)$. Note that there are $\gaussm{4}{1}{q}$ choices for $W_1$ and due to $W_1\le W_2$ there are $\gaussm{3}{1}{q}$ choices for $W_2$ when $W_1$ is fixed.} 
\end{proof}  
 

We remark that Proposition~\ref{prop_a_4_2} is tight for $q=2$, i.e., a corresponding code $\mathcal{C}$ of cardinality $105$ indeed exists. Such a code also exists if we prescribe 
a Singer cycle, i.e., a cyclic group of order $15$. Indeed, $15$ is the maximum possible order of the automorphism group (for $\#\mathcal{C}=105$). How to find such codes using 
integer linear programming, with or without prescribing automorphisms, is the topic of the next section.  
The underlying proof strategy of Proposition~\ref{prop_a_4_2} will be generalized in Section~\ref{sec_bounds}.

As usual in coding theory, the maximum cardinalities of codes can be lower and upper bounded by a canonical sphere covering and sphere packing bound, respectively. 
In the context of (full) flag codes the determination of the cardinalities of the \textit{spheres} is an open and non-trivial problem, see \cite{liebhold2018generalizing} 
for more details. Using the computational details on the sphere sizes determined in \cite{phd_liebhold} we determine the order of magnitude of the sphere packing and 
the sphere covering bound for $n\le 7$. In Table\ref{table_asymptotic_sphere_packing} we state exponents $e$ such that the sphere packing bound for $A_q^f(v,d)$ is 
$\Theta\!\left(q^e\right)$, i.e., we have lower and upper bounds for the sphere packing bound of the form $cq^e$ plus terms of lower order, where $c$ is a suitable constant.  
In Table~\ref{table_asymptotic_beta} we will summarize the exponents of the improved upper bounds obtained using the methods from this paper.  
The corresponding exponents for the sphere covering bound can be found in Table~\ref{table_asymptotic_sphere_covering}. \highlight{(For better comparison the two tables are 
located in Section~4.)}      

\section{An integer linear programming formulation for $A_q^f(v,d)$}
\label{sec_ILP}

In principle, it is rather simple to give an integer linear programming formulation for the exact determination of $A_q^f(v,d)$. Let us start with the formulation 
as a maximum independent set problem. To this end let $\mathcal{G}_{v,d,q}=(V,E)$ be a graph with vertex set $V=\mathcal{F}(v,q)$ and $\{\Lambda,\Lambda'\}\in E$ iff 
$\Lambda\neq\Lambda'$ and  $\fdist(\Lambda,\Lambda')<d$. Clearly, each flag code in $\mathbb{F}_q^v$ with minimum Grassmann distance $d$ is in bijection to an 
independent set in $\mathcal{G}_{v,d,q}$. A standard integer linear programming (ILP) formulation for the maximum cardinality of an independent set in a graph 
$(V,E)$ is given by $\max\sum_{u\in V} x_{\highlight{u}}$ subject to $x_u+x_w\le 1$ for all edges $\{u,w\}\in E$ and $x_u\in\{0,1\}$ for all $u\in V$. In our situation this 
gives:   
\begin{eqnarray}
  A_q^f(v,d) &=& \max \sum_{\Lambda\in\mathcal{F}(v,q)} x_\Lambda\quad\text{ s.t.}\label{ilp_target}\\
  x_\Lambda+x_{\Lambda'} &\le& 1\quad\quad\quad \forall \Lambda,\Lambda'\in \mathcal{F}(v,q)\text{ with }\Lambda\neq\Lambda', \fdist(\Lambda,\Lambda')<d\label{ilp_edge_constraint}\\
  x_\Lambda &\in &\{0,1\}\,\quad \forall\Lambda\in\mathcal{F}(v,q)\label{ilp_binary} 
\end{eqnarray}  
Note that the corresponding flag code is given by $\mathcal{C}=\left\{\Lambda\in\mathcal{F}(v,q)\mid x_\Lambda=1\right\}$ and that the formulation can be easily adopted for 
$A_q^f(v,d;T)$. The corresponding linear programming (LP) relaxation is obtained if the constraints from (\ref{ilp_binary}) are replaced by $0\le x_\Lambda\le 1$. Solving the 
LP relaxation, which is done by ILP solvers in intermediate steps, gives an upper bound. Since setting $x_\Lambda=\tfrac{1}{2}$ for all $\Lambda\in\mathcal{F}(v,q)$ always satisfies 
the constraints from (\ref{ilp_edge_constraint}), we cannot obtain an upper bound tighter than $\# \mathcal{F}(v,q)/2$ ($\# V/2$ in the general case), which is a rather bad bound 
(provided $d\ge 2$). However, for each subset $\mathcal{V}\subseteq V$ that induces a clique, i.e., $\{u,w\}$ is an edge for all pairs of different elements $u,w$ in $\mathcal{V}$, 
we can add the improved constraint $\sum_{u\in \mathcal{V}} x_u\le 1$, which is also called \emph{clique constraint}. \highlight{In many cases, adding such clique constraints results 
in a tighter LP upper bound.} So, the rest of this section is devoted to the description 
of large cliques in $\mathcal{G}_{v,d,q}$.   

For two vectors $x,y\in\mathbb{R}^n$ we write $x\le y$ iff $x_i\le y_i$ for all $1\le i\le n$. By $\mathbf{0}$ we denote the all zero vector whenever the length is clear from 
the context. 
We say that two subspaces $U,W$ of $\mathbb{F}_q^v$ are \emph{incident} if either $U\le W$ or $W\le U$, which we denote by $(U,W)\in I$.  

\begin{Lemma}
  \label{lemma_clique}
  Let $r\in\mathbb{N}^{v-1}$ with $\mathbf{0}\le r\le m(v)$, $\mathcal{I}=\left\{1\le i\le v-1\mid r_i\neq 0\right\}$, and $U_i$ an arbitrary subspace of $\mathbb{F}_q^v$ 
  with $\dim(U_i)\in\left\{i-\highlight{m(v)_i}+r_i,i+\highlight{m(v)_i}-r_i\right\}$ for each $i\in\mathcal{I}$. If $d> \sum_{i=1}^{v-1}\left(m(v)_i-r_i\right)$, then
  $$
    \mathcal{V}=\left\{ \left(W_1,\dots,W_{v-1}\right)\in\mathcal{F}(v,q) \mid \left(W_i,U_i\right)\in I\,\forall i\in\mathcal{I}  \right\}
  $$
  is the vertex set of a clique in $\mathcal{G}_{v,d,q}$.
\end{Lemma}
\begin{proof}
  Let $\Lambda=\left(W_1,\dots,W_{v-1}\right)$ and $\Lambda'=\left(W_1',\dots,W_{v-1}'\right)$ be two different elements in \highlight{$\mathcal{V}$}. For $1\le i\le v-1$ with 
  $i\notin \mathcal{I}$ we have $\idist(W_i,W_i')\le m(v)_i=m(v)_i-r_i$. Now we consider $i\in \mathcal{I}$.
  If $\dim(U_i)=i-\highlight{m(v)_i}+r_i$, then $U_i\le W_i$ and $U_i\le W_i'$, so that
  $$
    \idist(W_i,W_i')=i-\dim(W_i\cap W_i')\le i-\dim(U_i)=m(v)_i-r_i.  
  $$  
  If $\dim(U_i)=i+\highlight{m(v)_i}-r_i$, then $W_i\le U_i$ and $ W_i'\le U_i$, so that
  $$
    \idist(W_i,W_i')=\dim(W_i+ W_i')-i\le \dim(U_i)-i=m(v)_i-r_i.  
  $$   
  Thus, we have
  $$
    \fdist(\Lambda,\Lambda') \le \sum_{i=1}^{v-1}\left(m(v)_i-r_i\right)<d,
  $$
  i.e. $\{\Lambda,\Lambda'\}$ is an edge in $\mathcal{G}_{v,d,q}$.
\end{proof}

\begin{Corollary}
  \label{cor_clique}
  Let $r\in\mathbb{N}^{v-1}$ with $\mathbf{0}\le r\le m(v)$, $\mathcal{I}=\left\{1\le i\le v-1\mid r_i\neq 0\right\}$, and $U_i$ an arbitrary $\left(i-\highlight{m(v)_i}+r_i\right)$-space in 
  $\mathbb{F}_q^v$ for each $i\in\mathcal{I}$. If $d> \sum_{i=1}^{v-1}\left(m(v)_i-r_i\right)$, then
  $$
    \mathcal{V}=\left\{ \left(W_1,\dots,W_{v-1}\right)\in\mathcal{F}(v,q) \mid U_i\le W_i\,\forall i\in\mathcal{I}  \right\}
  $$
  is the vertex set of a clique in $\mathcal{G}_{v,d,q}$.
\end{Corollary}


The vector $r$ describes the reduction of the achievable Grassmann distance with respect to the maximum possible Grassmann distance. Let us consider an example, for $(v,d)=(4,2)$ 
we have $m(v)=(1,2,1)$ and $r=(1,2,0)$ satisfies the conditions of Corollary~\ref{cor_clique}, i.e., each full flag code $\mathcal{C}$ in $\mathbb{F}_q^4$ with minimum distance 
$\fdist(\mathcal{C})=2$ satisfies $\#\left\{\left(W_1,W_2,W_3\right)\in\mathcal{C}\mid W_1=P,W_2=L\right\} \le 1$ for each pair $(P,L)\in\gauss{4}{1}{q}\times \gauss{4}{2}{q}$. 
Actually, this argument was used in the proof of Proposition~\ref{prop_a_4_2} to conclude the upper bound for $A_q^f(4,2)$.

In the other direction, \highlight{a} strengthening of Corollary~\ref{cor_clique} is sufficient to cover all edges of $\mathcal{G}_{v,d,q}$ by corresponding cliques 
with vertex set $\mathcal{V}$.

\begin{Lemma}
  \label{lemma_other_direction}
  If $\Lambda=\left(W_1,\dots,W_{v-1}\right)$ and $\Lambda'=\left(W_1',\dots,\highlight{W_{v-1}'}\right)$ are two different full flags with 
  $\fdist(\Lambda,\Lambda')<d$, then there exist subspaces $U_1\le \dots\le U_{v-1}$ such that $d> \sum_{i=1}^{v-1}\left(m(v)_i-r_i\right)$ and 
  $\mathbf{0}\le r\le m(v)$, where $r_i=\dim(U_i)-i+m(v)_i$ for all $1\le i\le v-1$.
\end{Lemma}
\begin{proof}
  We choose $U_i=W_i\cap W_i'$ for all $1\le i\le v-1$, so that $U_1\le \dots\le U_{v-1}$. By construction we have
  $$
    \idist(W_i,W_i')=i-\dim(W_i\cap W_i')=i-\dim(U_i)=m(v)_i-r_i, 
  $$
  so that $\mathbf{0}\le r\le m(v)$ and $d>\fdist(\Lambda,\Lambda')=\sum_{i=1}^{v-1} \left(m(v)_i-r_i\right)$.
\end{proof}
In other words, we can replace the constraints (\ref{ilp_edge_constraint}) by the clique constraints $\sum_{u\in \mathcal{V}} x_u\le 1$ 
for all cases that satisfy the conditions of Corollary~\ref{cor_clique}, where we additionally assume $U_1\le\dots\le U_{v-1}$. In order to ease the 
notation we focus on the cliques of Corollary~\ref{cor_clique} instead of the more general situation of Lemma~\ref{lemma_clique}. 

\begin{Definition}
  For an integer vector $\mathbf{0}\le r\le m(v)$ let $\mathcal{I}=\left\{1\le i\le v-1\mid r_i>0\right\}$ and let $\mathcal{V}^r_{v,q}$ denote the set of cliques 
  $$
    \mathcal{V}=\left\{ \left(W_1,\dots,W_{v-1}\right)\in\mathcal{F}(v,q) \mid U_i\le W_i\,\forall i\in\mathcal{I}  \right\},
  $$   
  where the $U_i$ are $(i-\highlight{m(v)_i}+r_i)$-spaces and we have $U_i\le U_{i'}$ for all $i,i'\in\mathcal{I}$ with $i\le i'$. By $E^r_{v,q}$ we denote the set 
  of edges $e=\{\Lambda,\Lambda'\}$, where $e\subseteq \mathcal{V}$ for at least one $\mathcal{V}\in \mathcal{V}^r_{v,q}$. 
\end{Definition}

If $\mathbf{0}\le r\le r'\le m(v)$, then we obviously have $E^r_{v,q}\supseteq E^{r'}_{v,q}$. So, given $d$, it is sufficient to consider all $\mathcal{V}^r_{v,q}$ 
where $\sum_{i=1}^{v-1}\left(m(v)_i-r_i\right)=d-1$. \highlight{Note that $E^r_{v,q}=\emptyset$ is possible, e.g.\ for $r=(0,0,0,4,1,0,0)$.} In our example $(v,d)=(4,2)$ it suffices to 
consider the vectors $(1,2,0)$, $(1,1,1)$, and $(0,2,1)$. However, for $r=(1,1,1)$ we have $U_1\le U_2\le U_3$ with $\dim(U_1)=\dim(U_2)=1$, i.e., $U_1=U_2$, and 
$\dim(U_3)=3$. If $\Lambda=\left(W_1,W_2,W_3\right)$ and $\Lambda'=\left(W_1',W_2',W_3'\right)$ are flags with $U_1\le W_1$ and $U_1\le W_1'$, then $\idist(W_2,W_2')\le 1$ since 
$U_1\le W_2\cap W_2'$. In other words, also $\mathcal{V}^{(1,0,1)}_{v,q}$ consists of vertex sets of cliques in $\mathcal{G}_{4,2,q}$ \highlight{that cover the same edges 
as $\mathcal{V}^{(1,1,1)}_{v,q}$, cf.\ Lemma~\ref{lemma_r_bar}. Intuitively we may say that for two flags $\Lambda=\left(W_1,\dots,W_{v-1}\right)$ and 
 $\Lambda'=\left(W_1',\dots,W_{v-1}'\right)$ a relatively large intersection of $W_i$ and $W_i'$ implies a relatively large intersection of $W_{i+1}$ and $W_{i+1}'$ and 
 vice versa. This idea is made more precise in the next definition and Lemma~\ref{lemma_r_bar}.}  

\begin{Definition}
  \label{def_r_bar}
  Let $\mathbf{0}\le r\le m(v)$ and $u_j=\max\{2j-v,0\}+r_j$ for all $1\le j\le v-1$. Then, let
  $$
    \overline{u}_j=\max \highlight{\Big\{} \{u_i\mid 1\le i\le j \} \cup\{u_i-2(i-j)\mid j<i<v\} \highlight{\Big\}}
  $$ 
  and $\overline{r}_j=\overline{u}_j-j+m(v)_j$ for all $1\le j\le v$. With this, we set $\overline{r}=\left(\overline{r}_1,\dots,\overline{r}_{v-1}\right)$.
\end{Definition}

For further usage we state two easy lemmas without proof.
\begin{Lemma}
  \label{lemma_intersection_propagation}
  Let $W_a,W_a'$ be $a$-spaces and $W_b,W_b'$ be $\highlight{b}$-spaces in $\mathbb{F}_q^v$ with $W_a<W_b$ and $W_a'<W_b'$. 
  Then, we have $\dim(W_b\cap W_b')\ge\dim(W_a\cap W_a')$ and $\dim(W_a\cap W_a')\ge \dim(W_b\cap W_b')-2(b-a)$.
\end{Lemma}

\begin{Lemma}
  \label{lemma_chain_decrease}
  Let $U_1\le\dots\le U_n$ be a weakly increasing chain of subspaces in $\mathbb{F}_q^v$ and $u=\left(u_1,\dots,u_n\right)\in\mathbb{N}^n$ 
  satisfy $u_1\le \dots \highlight{\le} u_n$. If $\dim(U_i)\ge u_i$ for all $1\le i\le n$, then there exists a weakly increasing chain $U'_1\le\dots\le U'_n$ 
  of subspaces in $\mathbb{F}_q^v$ with $U_i'\le U_i$ and $\dim(U_i')=u_i$ for all $1\le i\le n$. 
\end{Lemma}

\begin{Lemma}
  \label{lemma_r_bar}
  For $\mathbf{0}\le r\le m(v)$ we have $r\le \overline{r}\le m(v)$ and $E^r_{v,q}=E^{\overline{r}}_{v,q}$. 
\end{Lemma}
\begin{proof}
  By construction we have $u_j=j-m(v)_j+r_j$ for $1\le j\le v-1$, since $j-m(v)_j=j-\min\{j,v-j\}=\max\{2j-v,0\}$. Setting $u=\left(u_1,\dots,u_{v-1}\right)$ 
  and $\overline{u}=\left(\overline{u}_1,\dots,\overline{u}_{v-1}\right)$, we note $u\le \overline{u}\le(1,\dots,v-1)$, so that 
  $r\le \overline{r}\le m(v)$ due to $\overline{r}_j=\overline{u}_j-j+m(v)_j$ for all $1\le j\le v$. From $r\le \overline{r}\le m(v)$ we conclude 
  $E^r_{v,q}\supseteq E^{\overline{r}}_{v,q}$. 
  
  Now let $\{\Lambda,\Lambda'\}\in E^r_{v,q}$, where $\Lambda=\left(W_1,\dots,W_{v-1}\right)$ and $\Lambda'=\left(W_1',\dots,\highlight{W_{v-1}'}\right)$. 
  We set $\mathcal{I}=\{1\le i\le v-1\mid r_i>0\}$ and note that the definition of $E^r_{v,q}$ yields the existence of an $u_i$-space $U_i$ in 
  $\mathbb{F}_q^v$ with $U_i\le W_i\cap W_i'$ for all $i\in \mathcal{I}$ and $U_i\le U_{i'}$ for all $i,i'\in\mathcal{I}$ with $i\le i'$. 
  Now we set $\bar U_j=W_j\cap W_j'$ for $j=1,\dots,v-1$. First we note $\dim(\bar U_j)\ge u_j$ for all $1\le j\le v-1$ and $\bar U_1\le\dots \highlight{\le}\bar U_{v-1}$. 
  Now let $1\le j\le v-1$ be fix but arbitrary. We want to show $\dim(\bar U_j)\ge \overline{u}_j$. If $\overline{u}_j=u_j$ this 
  is clearly the case. If $\overline{u}_j=u_h$ for an index $1\le h<j$, then we can choose $b=j$, $a=h$  in Lemma~\ref{lemma_intersection_propagation} 
  to conclude $\dim(\bar U_j)=\dim(W_j\cap W_j')\ge \dim(W_h\cap W_h')\ge u_h=\overline{u}_j$. If $\overline{u}_j=u_h-2(h-j)$ for an index $j<h<v$, then  
  we can choose $b=h$, $a=j$ in Lemma~\ref{lemma_intersection_propagation} to conclude $\dim(\bar U_j)=\dim(W_j\cap W_j')\ge \dim(W_h,W_h')-2(h-j)\ge u_h-2(h-j)=\overline{u}_j$. 
  Since $\bar{u}_1\le \dots\highlight{\le}\bar{u}_{v-1}$ by construction, we can apply Lemma~\ref{lemma_chain_decrease} to conclude the existence of 
  \highlight{subspaces} $U'_1\le\dots\le U'_{v-1}$ in $\mathbb{F}_q^v$ with $U'_j\le W_j\cap W_j'$ and $\dim(U'_j)=\bar{u}_j$ for all $1\le j\le v-1$. 
  Due to the definition of $\bar{r}$ this yields that $\{\Lambda,\Lambda'\}\in E^{\overline{r}}_{v,q}$. Since $\{\Lambda,\Lambda'\}\in E^r_{v,q}$ 
  was arbitrary, this gives $E^r_{v,q}\subseteq E^{\overline{r}}_{v,q}$, so that $E^r_{v,q}= E^{\overline{r}}_{v,q}$.     
\end{proof}

As an example we have $\overline{(1,0,1)}=(1,1,1)$, so that $E^{(1,1,1)}_{4,q}=E^{(1,0,1)}_{4,q}$. Here we have 
$\#\mathcal{V}=\gaussm{3}{1}{q}$ for each $\mathcal{V}\in \mathcal{V}^{(1,0,1)}_{4,q}$ and also $\#\mathcal{V}=\gaussm{3}{1}{q}$ 
for each $\mathcal{V}\in \mathcal{V}^{(1,1,1)}_{4,q}$. Moreover, $\# \mathcal{V}^{(1,0,1)}_{4,q}=\gaussm{4}{1}{q}\cdot \gaussm{3}{2}{q}=
\gaussm{4}{1}{q}\cdot \gaussm{3}{1}{q}=\#\mathcal{V}^{(1,1,1)}_{4,q}$. In other words, here, there is no difference at all between taking 
$\mathcal{V}^{(1,0,1)}_{4,q}$ or $\mathcal{V}^{(1,1,1)}_{4,q}$. However, for $v\ge 5$ improvements are possible, \highlight{in the sense that larger cliques give 
{\lq\lq}tighter{\rq\rq} (I)LP formulations that eventually decrease running times of the ILP solver. From the theoretical point of view we can state (without proof) that 
the bound of Theorem~\ref{thm_anticode_bound} applied to $r$ is at least as good as the bound applied to $\overline{r}$, which occurs 
in the required relation of the vector $r$ and the minimum distance $d$.} In general, we have $\#\mathcal{V}^r_{v,q}\le \#\mathcal{V}^{\overline{r}}_{v,q}$. 

\begin{Definition}
  For $a,b\in \left\{r\in\mathbb{N}^{v-1}\mid \mathbf{0}\le r\le m(v)\right\}$ we define $a\preceq b$ if either $\bar{a}<\bar{b}$ or 
  $\bar{a}=\bar{b}\,\highlight{\wedge}\,a\le b$. 
\end{Definition}

The conditions of a poset, i.e., reflexivity, antisymmetry, and transitivity, are directly verified. So each subset $\mathcal{R}\subseteq 
\left\{r\in\mathbb{N}^{v-1}\mid \mathbf{0}\le r\le m(v)\right\}$ contains a unique subset $\mathcal{R}'\subseteq\mathcal{R}$ of minimal elements, i.e., 
for each $r\in \mathcal{R}$ there exists an element $r'\in\mathcal{R}'$ with $r'\preceq r$ and there are no two different elements $r',r''\in\mathcal{R}'$ 
with $r'\preceq r''$. Moreover, $r\le r'$ implies $\bar{r}\le \bar{r'}$, so that $r\preceq r'$. However, the converse is not true as we will see in  
Example~\ref{ex_R_v_d}. More precisely, we have $(0,1,1,0)\preceq (1,0,1,0)$ while $(0,1,1,0)$ and $(1,0,1,0)$ are incomparable with respect to $\le$. 
(It is also easy to show that $\bar{\bar{r}}=\bar{r}$.)

\begin{Definition}
  Let $\mathcal{R}_{v,d}$ \highlight{be} the unique set of, with respect to $\preceq$, minimal elements in the set of vectors
  $$
    \left\{r\in\mathbb{N}^{v-1}\mid \mathbf{0}\le r\le m(v),\, d>\sum_{i=1}^{v-1}\left(m(v)_i-\bar{r}_i\right)\right\}.
  $$  
\end{Definition}
Note that $r\in \mathcal{R}_{v,d}$ implies $\sum_{i=1}^{v-1}\left(m(v)_i-r_i\right)<d$ and $\left(r_1,\dots,r_{v-1}\right)\in\mathcal{R}_{v,d}$ 
if and only if $\left(r_{v-1},\dots,r_1\right)\in\mathcal{R}_{v,d}$. 

\begin{Example}
  \label{ex_R_v_d}
  For $v=d=5$ the vectors in $\left\{r\in\mathbb{N}^{v-1}\mid \mathbf{0}\le r\le m(v),\, d-1=\sum_{i=1}^{v-1}\left(m(v)_i-r_i\right)\right\}$ 
  are given by $(0,2,0,0)$, $(0,0,2,0)$, $(1,1,0,0)$, $(1,0,1,0)$, $(1,0,0,1)$, $(0,1,1,0)$, $(0,1,0,1)$, and $(0,0,1,1)$. We remark that
  $\overline{(1,0,0,0)}=(1,1,0,0)$, $\overline{(0,1,0,0)}=(0,1,0,0)$, $\overline{(1,1,0,0)}=(1,1,0,0)$, $\overline{(1,0,1,0)}=(1,1,1,0)$, $\overline{(1,0,0,1)}=(1,1,1,1)$, 
  $\overline{(0,1,1,0)}=(0,1,1,0)$, and $\overline{(0,2,0,0)}=(0,2,1,0)$. Since $\overline{(1,0,1,0)}=(1,1,1,0)> (0,1,1,0)=\overline{(0,1,1,0)}$, we e.g.\ have 
  $(1,0,1,0)\notin\mathcal{R}_{5,5}$. Similarly we have $(0,2,0,0)\notin\mathcal{R}_{5,5}$ since $\overline{(0,2,0,0)}=(0,2,1,0)> (0,1,1,0)=\overline{(0,1,1,0)}$. 
  \highlight{After performing all pairwise comparisons we end up with
  $$
    \mathcal{R}_{5,5}=\Big\{(1,0,0,0),(0,1,1,0),(0,0,0,1)\Big\}.
  $$} 
\end{Example}

\begin{Proposition}
\label{prop_ILP}
\begin{eqnarray}
  A_q^f(v,d) &=& \max \sum_{\Lambda\in\mathcal{F}(v,q)} x_\Lambda\quad\text{ s.t.}\label{ilp_target_impr}\\
  \sum_{\Lambda\in\mathcal{V}} x_{\Lambda} &\le& 1\quad\quad\quad \forall \mathcal{V}\in \mathcal{V}^r_{v,q}\, \forall r\in \mathcal{R}_{v,d}\label{ilp_clique_constraint}\\
  x_\Lambda &\in &\{0,1\}\,\quad \forall\Lambda\in\mathcal{F}(v,q)\label{ilp_binary_impr} 
\end{eqnarray}  
\end{Proposition}
\begin{proof}
  We start from the ILP formulation (\ref{ilp_target})-(\ref{ilp_binary}). Now let $\Lambda,\Lambda'\in \mathcal{F}(v,q)$ with $\Lambda\neq \Lambda'$ and $\fdist(\Lambda,\Lambda')<d$. 
  From Lemma~\ref{lemma_other_direction} and Corollary~\ref{cor_clique} we conclude the existence of a vector $\mathbf{0}\le r'\le m(v)$ with $\{\Lambda,\Lambda'\}\in E^{r'}_{v,q}$, 
  which is contained in the edge set of $\mathcal{G}_{v,d,q}$. W.l.o.g.\ we can additionally assume that $d-1=\sum_{i=1}^{v-1} \left(m(v)_i-r_i'\right)$. From 
  Lemma~\ref{lemma_r_bar} 
  we then conclude the existence of $r\in \mathcal{R}_{v,d}$ with $E^{r}_{v,q}=E^{r'}_{v,q}$. 
  
  It remains to remark that for each $\mathcal{V}\in \mathcal{V}^r_{v,q}$ and each $r\in \mathcal{R}_{v,d}$ constraint~(\ref{ilp_clique_constraint}) is a valid constraint due to
  Lemma~\ref{lemma_r_bar} 
  and Corollary~\ref{cor_clique}.  
\end{proof}

\begin{table}[htp]
  \begin{center}
    \begin{tabular}{cc}
      \hline
      $(v,d)$ & $\mathcal{R}_{v,d}$ \\ 
      \hline
      $(5,1)$ & $\Big\{(1,2,2,1)\Big\}$\\
      $(5,2)$ & $\Big\{(1,2,2,0),(1,2,0,1), (1,0,2,1), (0,2,2,1)\Big\}$\\
      $(5,3)$ & $\Big\{(1,2,0,0), (1,0,2,0), (0,2,2,0), (1,0,0,1), (0,2,0,1), (0,0,2,1)\Big\}$\\
      $(5,4)$ & $\Big\{(1,0,1,0), (0,2,0,0), (0,0,2,0), (0,1,0,1)\Big\}$\\
      $(5,5)$ & $\Big\{(1,0,0,0), (0,1,1,0), (0,0,0,1)\Big\}$\\
      $(5,6)$ & $\Big\{(0,1,0,0), (0,0,1,0)\Big\}$\\ 
      $(6,1)$ & $\Big\{(1,2,3,2,1)\Big\}$\\      
      $(6,2)$ & $\Big\{(1,2,3,2,0), (1,2,3,0,1), (1,2,0,2,1), (1,0,3,2,1), (0,2,3,2,1)\Big\}$\\
      $(6,3)$ & $\Big\{(1,2,3,0,0), (1,2,0,2,0), (1,2,0,0,1), (1,0,3,2,0), (1,0,3,0,1), (1,0,0,2,1), (0,2,3,2,0),$\\ 
              & $(0,2,3,0,1), (0,2,0,2,1), (0,0,3,2,1)\Big\}$\\
      $(6,4)$ & $\Big\{(1,2,0,1,0), (1,0,3,0,0), (1,0,2,0,1), (1,0,0,2,0), (0,2,3,0,0), (0,2,0,2,0), (0,2,0,0,1),$\\ 
              &  $(0,1,0,2,1), (0,0,3,2,0), (0,0,3,0,1)\Big\}$\\
      $(6,5)$ & $\Big\{(1,2,0,0,0), (1,0,2,1,0), (1,0,0,0,1), (0,2,0,1,0), (0,1,2,0,1), (0,1,0,2,0), (0,0,3,0,0),$\\ 
              & $(0,0,0,2,1)\Big\}$\\
      $(6,6)$ & $\Big\{(1,0,2,0,0), (1,0,0,1,0), (0,2,0,0,0), (0,1,2,1,0), (0,1,0,0,1), (0,0,2,0,1), (0,0,0,2,0)\Big\}$\\
      $(6,7)$ & $\Big\{(1,0,0,0,0), (0,1,2,0,0), (0,1,0,1,0), (0,0,2,1,0), (0,0,0,0,1)\Big\}$\\
      $(6,8)$ & $\Big\{(0,1,0,0,0), (0,0,2,0,0), (0,0,0,1,0)\Big\}$\\
      $(6,9)$ & $\Big\{(0,0,1,0,0)\Big\}$\\
      \hline      
    \end{tabular}
    \caption{The sets $\mathcal{R}_{v,d}$ for small parameters.}
  \end{center}
\end{table}      

Due to combinatorial explosion, the number of variables and constraints of the ILP from Proposition~\ref{prop_ILP} gets large even for small parameters. 
So, in order to construct large flag codes we want to reduce the computational complexity by prescribing automorphisms \highlight{-- a technique that is widely 
used for the construction of many combinatorial objects.} An \emph{automorphism} $\varphi$ of 
$\mathcal{C}=\left\{\Lambda_1,\dots,\Lambda_m\right\}\subseteq \mathcal{F}(v,q)$ is an element of $\operatorname{GL}(v,q)$ such that $\mathcal{C}=
\left\{\varphi(\Lambda_1),\dots,\varphi(\Lambda_m)\right\}$. By $\operatorname{Aut}(\mathcal{C})$ we denote the group of automorphisms of $\mathcal{C}$, 
which is a subgroup of $\operatorname{GL}(v,q)$. For notational reason we rewrite the ILP from Proposition~\ref{prop_ILP} to 
$\max \sum_{\Lambda\in\mathcal{F}(v,q)} x_\Lambda$ subject to $Mx\le \mathbf{1}$, where the $x_i$ are binary variables, $\mathbf{1}$ is the all-$1$ vector, and 
$$
  M_{\mathcal{V},\Lambda}=\left\{\begin{array}{cc}
  1 & \text{if } \Lambda\in\mathcal{V},\\ 
  0 & \text{otherwise}
  \end{array}\right.
$$   
for all $\Lambda\in\mathcal{F}(v,q)$ and all $\mathcal{V}\in \mathcal{V}_{v,q}^r$, $r\in\mathcal{R}_{v,d}$. 

Now let $G\le \operatorname{Aut}(\mathcal{C})\le \operatorname{GL}(v,q)$. By $M^G$ we denote the corresponding matrix briefly defined below, see e.g.\ 
\cite{kohnert2008construction} where the method was applied to \emph{constant dimension codes}, i.e., flag codes with type $T$, where $\# T=1$. The underlying 
general method can be described as follows. In order to obtain $M^G$, the matrix $M$ is reduced by adding up  columns (labeled by the flags contained in 
$\mathcal{F}(v,q)$) corresponding to the orbits of $G$, which we denote by $\omega_1,\dots,\omega_\gamma$. Due \highlight{to} the equivalence
\begin{equation}
  \label{property_automorphism_incidences}
  U\le W\quad\Longleftrightarrow\quad \varphi(U)\le\varphi(W)
\end{equation}
for all subspaces $U,W$ of $\mathbb{F}_q^v$ and each automorphism $\varphi\in G$ we have that rows corresponding to vertex sets $\mathcal{V}$, $\mathcal{V}'$ 
in the same orbit under $G$ are equal. Therefore the redundant rows are removed from the matrix and we obtain a smaller matrix denoted by $M^G$. The number of 
rows of $M^G$ is then the number $\Gamma$ of orbits of $G$ on $\left\{\mathcal{V}\mid \mathcal{V}\in \mathcal{V}_{v,q}^r, r\in\mathcal{R}_{v,d}\right\}$, which we denote 
by $\Omega_1,\dots,\Omega_\Gamma$. The number $\gamma$ of columns of $M^G$ is the number of orbits of $G$ on the flags in $\mathcal{F}(v,q)$. For an entry of $M^G$ we have
$$
  M_{\Omega_i,\omega_j}=\#\left\{ \Lambda\in\omega_j \mid \Lambda\in \mathcal{V}\right\},
$$
where $\mathcal{V}$ is a representative of the orbit $\Omega_i$.  Because of property~(\ref{property_automorphism_incidences}) the matrix $M^G$ is well-defined 
as the definition of $M^G_{\Omega_i,\omega_j}$ is independent of the representative $\mathcal{V}$. Thus, we can restate Proposition~\ref{prop_ILP} as follows:
\begin{Theorem}
  \label{thm_ILP}
  Let $G$ be a subgroup of $\operatorname{GL}(v,q)$. There is a flag code $\mathcal{C}\subseteq \mathcal{F}(v,q)$ with minimum Grassmann distance $d$ whose 
  group of automorphisms contains $G$ as a subgroup if, and only if, there is a $(0/1)$-solution $x=\left(x_1,\dots,x_\gamma\right)^\top$ satisfying 
  $\#\mathcal{C}=\sum_{i=1}^\gamma \left|\omega_i\right|\cdot x_i$ and $M^gx\le\mathbf{1}$.
\end{Theorem}

Note that $M_{\Omega_i,\omega_j}>1$ implies $x_{\omega_j}=0$. However, those conclusions are automatically drawn in a preprocessing step by the most commonly used 
ILP solvers.

\begin{Example}
  \label{ex_prescribe_automorphims} 
  We want to apply Theorem~\ref{thm_ILP} in order to obtain lower bounds for $A_2^f(5,2)$. Without prescribing automorphisms there are $\#\mathcal{F}(5,2)=9765$ full flags, 
  i.e., variables, and $13020 \highlight{=4\cdot 3255}$ constraints, since $\#\mathcal{V}^{(1,2,2,0)}_{5,2}=\#\mathcal{V}^{(1,2,0,1)}_{5,2}=\#\mathcal{V}^{(1,0,2,1)}_{5,2}=\#\mathcal{V}^{(0,2,2,1)}_{5,2} 
  =3255$. We prescribe a group $G$ of automorphisms generated by a single element:
  $$
    G:=\left\langle
    \begin{pmatrix}
      0 & 1 & 0 & 0 & 0\\
      0 & 0 & 1 & 0 & 0\\
      0 & 0 & 0 & 1 & 0\\
      0 & 0 & 0 & 0 & 1\\
      1 & 0 & 1 & 1 & 1\\
    \end{pmatrix}    
    \right\rangle.
  $$ 
  $G$ is a cyclic group of order $31$ -- indeed it is a Singer group. The reduced ILP consists of $420$ constraints and $315$ binary variables. Using the ILP solver 
  \texttt{ILOG CPLEX}\footnote{\highlight{\url{https://www.ibm.com/de-de/products/ilog-cplex-optimization-studio}}} an optimal solution with target value\footnote{\highlight{
  The target value of a feasible solution of an optimization problem is the value of the function that is optimized evaluated at that point. In the ILP of 
  Proposition~\ref{prop_ILP} the target function is the sum on the right hand side of (\ref{ilp_target_impr}).}} $3069$ was found after 213~seconds of computation time and 
  68\,180 branch-\&-bound nodes. \highlight{Thus, we can conclude $A_2^f(5,2)\ge 3069$.}  
  The group given by 
  $$
    G:=\left\langle
    \begin{pmatrix}
    1 & 0 & 0 & 0 & 0\\
    0 & 0 & 1 & 0 & 0\\
    0 & 0 & 0 & 1 & 0\\
    0 & 0 & 0 & 0 & 1\\
    0 & 1 & 1 & 0 & 0
    \end{pmatrix}    
    \right\rangle
  $$
  is a cyclic group of order $15$, and indeed a Singer group of a hyperplane. The corresponding reduced ILP consists of $865$~constraints and $651$ binary variables. 
  After 11~minutes and 24\,895 branch-\&-bound nodes a flag code with cardinality $3120$ was found, \highlight{so that we can conclude 
  $A_2^f(5,2)\ge 3120$}. After 9~hours and 6\,799\,282 branch-\&-bound nodes the upper 
  bound dropped to 3178 while no better solution was found. \highlight{So, possibly  a code $\mathcal{C}$ with cardinality $3120<\#\mathcal{C}\le 3178$ might be found 
  if we give the ILP solver more time to finish the computation. Although we have aborted the computation, we can still draw the conclusion that there is no code 
  of cardinality strictly larger than $3178$ that admits $G$ as a subgroup of its automorphisms. However, this does not give an upper bound for $A_2^f(5,2)$ at all.} 
  For a cyclic group of order $15$ we found that the \highlight{optimal} target value\footnote{\highlight{By an optimal 
  target value we denote the target value that is attained in the extremum, i.e., the maximum or minimum depending on the formulation of the optimization problem.}} lies between 
  $2982$ and  $3068$. Since already the upper bound is strictly less than the cardinality of the best known solution we have aborted the solution process.
  
  Performing a more extensive computational experiment we remark that there are several groups where we can easily verify that the corresponding upper bound is 
  strictly less than $3120$, \highlight{i.e. prescribing such groups will not give us better codes}. An example is given by the matrix 
  $$
    \begin{pmatrix}
      1 & 0 & 0 & 0 & 0\\
      0 & 1 & 0 & 0 & 0\\
      0 & 0 & 1 & 0 & 0\\
      0 & 0 & 0 & 1 & 1\\
      0 & 0 & 0 & 0 & 1                                        
    \end{pmatrix},
  $$ 
  which generates a group of order $2$, has $116$ fix points, and which does not allow a flag code with cardinality strictly larger than $2807$. Examples 
  of small groups where the achievable cardinality is strictly smaller than $3255$, \highlight{i.e.\ candidates for groups that possibly may yield better 
  codes than currently known but definitely cannot reach the best known upper bound for $A_2^f(5,2)$,} are given by the matrices
  $$
    \begin{pmatrix}
      1 & 0 & 0 & 0 & 0\\
      0 & 1 & 0 & 0 & 0\\
      0 & 0 & 1 & 0 & 0\\
      0 & 0 & 0 & 0 & 1\\
      0 & 0 & 0 & 1 & 1  
    \end{pmatrix}
    \quad\text{and}\quad
    \begin{pmatrix}
      1 & 0 & 0 & 0 & 0\\
      0 & 1 & 1 & 0 & 0\\
      0 & 0 & 1 & 0 & 0\\
      0 & 0 & 0 & 1 & 1\\
      0 & 0 & 0 & 0 & 1                                        
    \end{pmatrix},
  $$   
  which generate cyclic groups of orders $3$ or $2$, have $30$ or $52$ fix points, and where we have upper bounds on the cardinality of $3171$ or $3144$, respectively. 
  Examples of cyclic groups where the ILP approach did not bring the upper bound strictly below $3255$, \highlight{i.e. which still might allow codes matching the 
  known upper bound $A_2^f(5,2)\le 3225$ from Proposition~\ref{prop_a_5_2}}, after a reasonable computation time are given by the matrices
  $$
    \begin{pmatrix}
      1 & 0 & 0 & 0 & 0\\
      0 & 0 & 1 & 0 & 0\\
      0 & 1 & 1 & 0 & 0\\
      0 & 0 & 0 & 0 & 1\\
      0 & 0 & 0 & 1 & 1
    \end{pmatrix},
    \begin{pmatrix}
      1 & 0 & 0 & 0 & 0\\
      0 & 1 & 0 & 0 & 0\\
      0 & 0 & 0 & 1 & 0\\
      0 & 0 & 0 & 0 & 1\\
      0 & 0 & 1 & 0 & 1
    \end{pmatrix}, 
    \begin{pmatrix}
      1 & 0 & 0 & 0 & 0\\
      0 & 1 & 0 & 0 & 0\\
      0 & 0 & 0 & 1 & 0\\
      0 & 0 & 0 & 0 & 1\\
      0 & 0 & 1 & 1 & 0
    \end{pmatrix},\text{ and }
    \begin{pmatrix}
      1 & 0 & 0 & 0 & 0\\
      0 & 0 & 1 & 0 & 0\\
      0 & 0 & 0 & 1 & 0\\
      0 & 0 & 0 & 0 & 1\\
      0 & 1 & 1 & 1 & 1
    \end{pmatrix}.
  $$
  The corresponding orders are $3$, $7$, $7$, and $5$, respectively. ($12$, $0$, $8$, and $2$ fix points.) 
  \highlight{To sum up, we have $3120\le A_2^f(5,2)\le 3225$, where only the lower bound is obtained with ILP computations and the stated upper bound is given by 
  Proposition~\ref{prop_a_5_2}.}  
\end{Example}  
\highlight{A concrete example of a flag code described by orbit representatives is stated directly after Proposition~\ref{prop_a_5_4}.}

We remark that the ILP formulations from Proposition~\ref{prop_ILP} and Theorem~\ref{thm_ILP} can be enhanced by additional  
bounds for substructures of flag codes. Examples are the bounds from Proposition~\ref{prop_johnson_point} and Proposition~\ref{prop_cdc} 
in the subsequent Section~\ref{sec_bounds}.

\section{\highlight{Upper bounds}}  
\label{sec_bounds}

In this section we want to generalize the idea underlying the upper bound of Proposition~\ref{prop_a_4_2} for $A_q^f(4,2)$, see Theorem~\ref{thm_anticode_bound}.
It will turn out that this can be seen as a generalization of the anticode bound for constant dimension codes \highlight{\cite[Theorem 5.2]{wang2003linear}}. In 
Proposition~\ref{prop_johnson_point}  we follow the approach of the Johnson bound for constant dimension codes \highlight{\cite[Theorem 2]{xia2009johnson}}. 
Together with Proposition~\ref{prop_cdc} we determine a general explicit upper bound of the 
form $A_q^f(v,d)\le q^\beta+O\!\left(q^{\beta-1}\right)$, see Proposition~\ref{prop_asymptotic_upper_bound}.

\begin{Definition}
  \label{def_weakly_increasing}
  Let $\mathcal{I}\subseteq \mathbb{N}$ and $U_i\le \mathbb{F}_q^v$ for all $i\in\mathcal{I}$. We call $\left(U_i\right)_{i\in\mathcal{I}}$ \emph{weakly increasing} 
  if $U_i\le U_j$ for all $i,j\in\mathcal{I}$ with $i\le j$.
\end{Definition}

\begin{Theorem}
  \label{thm_anticode_bound}
  Let $\mathbf{0}\le r\le m(v)$ with $d>\sum_{i=1}^{v-1} \left(m(v)_i-\overline{r}_i\right)$ and 
  $\mathcal{I}=\{1\le i\le v-1\mid r_i>0\}$. Then, we have $A_q^f(v,d)\le \#\mathcal{U} / \#\widehat{\mathcal{U}}$, where
  \begin{eqnarray*}
    \mathcal{U} &=&\left\{ \left(U_i\right)_{i\in\mathcal{I}} \text{ weakly increasing }\mid \dim(U_i)=i-m(v)_i+r_i \,\forall i\in\mathcal{I}\right\},\\ 
    \widehat{\mathcal{U}} &=& \left\{ \left(U_i\right)_{i\in\mathcal{I}} \text{ weakly increasing }\mid \dim(U_i)=i-m(v)_i+r_i, U_i\le W_i' \,\forall i\in\mathcal{I} \right\},
  \end{eqnarray*}
  and $\Lambda'=\left(W_1',\dots,W_{v-1}'\right)\in\mathcal{F}(v,q)$ is an arbitrary but fixed full flag. 
  
  \highlight{Setting $u_i=i-m(v)_i+r_i$ for $1\le i\le v-1$ and $\mathcal{I}=\{1\le i\le v-1\mid r_i>0\}=
  \left\{k_1,\dots,k_m\right\}$, where $0<k_1<\dots<k_m<v$, we have
  $$
    \frac{\#\mathcal{U}}{\#\widehat{\mathcal{U}}}=\frac{\gaussm{v}{u_{k_1}}{q}\cdot \prod_{i=2}^{m} \gaussm{v-u_{k_{i-1}}}{u_{k_i}-u_{k_{i-1}}}{q}}
    {\gaussm{k_1}{u_{k_1}}{q}\cdot \prod_{i=2}^m \gaussm{ k_i-u_{k_{i-1}} }{u_{k_i}-u_{k_{i-1}} } {q}}.
  $$}
\end{Theorem}
\begin{proof}
  Let $\mathcal{C}$ be a full flag code in $\mathbb{F}_q^v$ with minimum Grassmann distance $d$. From Corollary~\ref{cor_clique} and Lemma~\ref{lemma_r_bar} we conclude
  $$
    \#\left\{ \left(W_1,\dots,W_{v-1}\right)\in\mathcal{C} \mid U_i\le W_i \right\}\le 1
  $$ 
  for each $\left(U_i\right)_{i\in\mathcal{I}}\in\mathcal{U}$. If $\left(W_1',\dots,W_{v-1}'\right)\in\mathcal{C}$ is arbitrary but fixed, then there are exactly 
  $\#\widehat{\mathcal{U}}$ elements $\left(U_i\right)_{i\in\mathcal{I}}\in\mathcal{U}$ with $U_i\le W_i'$ for all $i\in \mathcal{I}$ since $\#\widehat{\mathcal{U}}$ 
  is independent of the choice of $\Lambda'$ \highlight{as we will see in the remaining counting part.} 
  
  From Equation~(\ref{eq_count_flags}) we conclude 
  $$
    \#\mathcal{U}=\gaussm{v}{u_{k_1}}{q}\cdot \prod_{i=2}^{m} \gaussm{v-u_{k_{i-1}}}{u_{k_i}-u_{k_{i-1}}}{q}.
  $$
  If $A\le B$ are two subspaces in $\mathbb{F}_q^v$, then the number of subspaces $X$ with $A\le X\le B$ with dimension $\dim(A)\le x\le\dim(B)$ is given by 
  $\gaussm{\dim(B)-\dim(A)}{x-\dim(A)}{q}$. Thus, we can iteratively conclude
  $$
    \#\widehat{\mathcal{U}}=\gaussm{k_1}{u_{k_1}}{q}\cdot \prod_{i=2}^m \gaussm{ k_i-u_{k_{i-1}} }{u_{k_i}-u_{k_{i-1}} } {q}.
  $$  
\end{proof}

We remark that Theorem~\ref{thm_anticode_bound} generalizes the \emph{anticode bound} for constant dimension subspace codes \highlight{\cite[Theorem 5.2]{wang2003linear}}, i.e., 
$$
A_q^i(v,d;k)\le \gaussm{v}{k-d+1}{q}/\gaussm{k}{k-d+1}{q}.
$$

\begin{Example}
  \highlight{In order to obtain upper bounds for $A_q^f(6,7)$ and $A_q^f(6,6)$, we apply Theorem~\ref{thm_anticode_bound} for the vectors $r\in\mathcal{R}_{v,d}$.}   
  \begin{itemize}
    \item For $r=(1,0,0,0,0)$ Theorem~\ref{thm_anticode_bound} gives $A_q^f(6,7)\le \gaussm{6}{1}{q}=q^5+q^4+q^3+q^2+q+1$.
    \item For $r=(0,1,2,0,0)$ Theorem~\ref{thm_anticode_bound} gives $A_q^f(6,7)\le \gaussm{6}{1}{q}\cdot \gaussm{5}{1}{q}/\gaussm{2}{1}{q}/\gaussm{2}{1}{q}=q^7 + 2q^5 
          + 3q^3 - q^2 + 3q - 2+\frac{3}{q+1}$.
    \item For $r=(0,1,0,1,0)$ Theorem~\ref{thm_anticode_bound} gives $A_q^f(6,7)\le \gaussm{6}{1}{q}\cdot\gaussm{5}{2}{q}/\gaussm{2}{1}{q}/\gaussm{3}{2}{q}=
          q^8 + 2q^6 + q^5 + 2q^4 + q^3 + 2q^2 + 1$.
    \item For $r=(1, 0, 2, 0, 0)$ Theorem~\ref{thm_anticode_bound} gives $A_q^f(6,6)\le\gaussm{6}{1}{q}\cdot\gaussm{5}{1}{q}/\gaussm{2}{1}{q}=
          \gaussm{6}{2}{q}=q^8 + q^7 + 2q^6 + 2q^5 + 3q^4 + 2q^3 + 2q^2 + q + 1$.            
    \item For $r=(1, 0, 0, 1, 0)$ Theorem~\ref{thm_anticode_bound} gives $A_q^f(6,6)\le \gaussm{6}{1}{q}\cdot\gaussm{5}{2}{q}/\gaussm{3}{1}{q}=
              q^9 + q^8 + 2q^7 + 3q^6 + 3q^5 + 3q^4 + 3q^3 + 2q^2 + q + 1$.
    \item For $r=(0, 2, 0, 0, 0)$ Theorem~\ref{thm_anticode_bound} gives $A_q^f(6,6)\le\gaussm{6}{2}{q}=q^8 + q^7 + 2q^6 + 2q^5 + 3q^4 + 2q^3 + 2q^2 + q + 1$.
    \item For $r=(0, 1, 2, 1, 0)$ Theorem~\ref{thm_anticode_bound} gives $A_q^f(6,6)\le\gaussm{6}{1}{q}\cdot\gaussm{5}{1}{q}\cdot \gaussm{4}{1}{q}/\gaussm{2}{1}{q}^3=
          q^9 + 3q^7 + 5q^5 - q^4 + 6q^3 - 3q^2 + 6q - 5+\frac{6}{q+1}$.
  \end{itemize}
  \highlight{So, different choices for $r$ may result in different bounds. Note that for each $r\in\mathcal{R}_{v,d}$ also the vector $r':=\left(r_{v-1},\dots,r_1\right)$ 
  is contained in $\mathcal{R}_{v,d}$, e.g.\ $r=(1, 0, 0, 1, 0)$ and $r'=(0, 1, 0, 0, 1)$. Applying Theorem~\ref{thm_anticode_bound} to $r'$ gives $A_q^f(6,6)\le 
  \gaussm{6}{1}{q}\cdot\gaussm{5}{1}{q}/\gaussm{2}{1}{q}=\gaussm{6}{2}{q}=q^8 + q^7 + 2q^6 + 2q^5 + 3q^4 + 2q^3 + 2q^2 + q + 1$. Each such pair $r,r'$ leads to the same 
  upper bound, which is explained by duality. So, our above enumeration of roughly half of the elements of $\mathcal{R}_{6,7}$ and $\mathcal{R}_{6,6}$ is sufficient 
  to find the tightest possible upper bound based on Theorem~\ref{thm_anticode_bound} cf.\ Proposition~\ref{prop_a_6_6} and Proposition~\ref{prop_a_6_7}.  
  Note that the bound for $A_q^f(6,6)$ can be concluded from the vectors $r=(0, 2, 0, 0, 0)$ and $r=(1, 0, 2, 0, 0)$, which is not explained 
  by duality.}
\end{Example}

We summarize these examples \highlight{to} the following two upper bounds.
\begin{Proposition}
  \label{prop_a_6_6}
  $$A_q^f(6,6)\le\gaussm{6}{2}{q}=q^8 + q^7 + 2q^6 + 2q^5 + 3q^4 + 2q^3 + 2q^2 + q + 1$$
\end{Proposition}
\begin{proof}
  We apply Theorem~\ref{thm_anticode_bound} with $r=(0,2,0,0,0)$ noting that $\overline{r}=(0,2,2,0,0)$.
\end{proof}
\highlight{For $q=2$ Proposition~\ref{prop_a_6_6} gives $A_2^f(6,6)\le 651$. Let $g_6$ be a generator of a Singer group in $\mathbb{F}_2^6$, i.e., a 
cyclic group of order $63$. Then $g_6^9$ is a generator of a cyclic group of order $7$.} If we prescribe the cyclic group of order $7$ generated by $g_6^9$, \highlight{i.e.,}  
\highlight{$$
  g_6^9:=\begin{pmatrix}
   0 & 1 & 0 & 0 & 0 & 0\\
   0 & 0 & 1 & 0 & 0 & 0\\
   1 & 1 & 0 & 0 & 0 & 0\\
   0 & 0 & 0 & 0 & 1 & 0\\
   0 & 0 & 0 & 0 & 0 & 1\\
   0 & 0 & 0 & 1 & 1 & 0
  \end{pmatrix}
$$}
for $q=2$, then the corresponding ILP of Theorem~\ref{thm_ILP} 
admits a solution of cardinality $224$, while we aborted the solution process before it was finished. \highlight{So, we have $224\le A_2^f(6,6)\le 651$ 
and at least one of the bounds is rather weak. Later on we improve the upper bound to $A_2^f(6,6)\le 567$, see Corollary~\ref{cor_a_6_6_j}.}

\begin{Proposition}
  \label{prop_a_6_7}
  $$A_q^f(6,7)\le\gaussm{6}{1}{q}=q^5+q^4+q^3+q^2+q+1$$
\end{Proposition}
\begin{proof}
  We apply Theorem~\ref{thm_anticode_bound} with $r=(1,0,0,0,0)$ noting that $\overline{r}=(1,1,1,0,0)$.
\end{proof}
\highlight{For $q=2$ Proposition~\ref{prop_a_6_7} gives $A_2^f(6,7)\le 63$. } 
If we prescribe the cyclic group of order $7$ \highlight{generated by $g_6^9$, see above,} for $q=2$, then the corresponding ILP of Theorem~\ref{thm_ILP} admits a solution of cardinality $63$, which was found in the root node. 
\highlight{Thus, we have $A_q^f(6,7)=63$.} 

Another example of the application of Theorem~\ref{thm_anticode_bound} is given by:
\begin{Proposition}
  \label{prop_a_6_8}
  $$A_q^f(6,8)\le\gaussm{6}{1}{q} \highlight{/} \gaussm{2}{1}{q}=q^4+q^2+1$$
\end{Proposition}
\begin{proof}
  We apply Theorem~\ref{thm_anticode_bound} with $r=(0,1,0,0,0)$ noting that $\overline{r}=(0,1,1,0,0)$.
\end{proof}

We remark that applying Theorem~\ref{thm_anticode_bound} with $r=(0,0,2,0,0)$ gives $A_2^f(6,8)\le \gaussm{6}{2}{2}/\gaussm{3}{2}{2}=93$. However, we also get 
$A_2^f(6,8)\le A_2^i(6,2;3)=77$ from $r=(0,0,2,0,0)$, which is of course superseded by $r=(0,1,0,0,0)$, \highlight{where Proposition~\ref{prop_a_6_8} 
yields $A_2^f(6,8)\le 21$}. For $q=2$, solving the ILP from Proposition~\ref{prop_ILP} 
directly gives a full flag code of matching cardinality $21$ after 35~minutes and 2577 branch-\&-bound nodes. If we prescribe the cyclic group of order $7$ 
\highlight{generated by $g_6^9$, see above,} 
then the corresponding ILP of Theorem~\ref{thm_ILP} admits a solution of cardinality $21$, which was found in the root node. \highlight{Thus, we have $A_2^f(6,8)=21$.} 

For constant dimension codes the anticode bound was improved to the so-called \emph{Johnson bound} $A_q^i(v,d;k)\le \left\lfloor\frac{q^v-1}{q^k-1}\cdot A_q^i(v-1,d;k-1)\right\rfloor$ 
for the cases where $d<k$, \highlight{see \cite[Theorem 2]{xia2009johnson}}. More precisely, without rounding down the iterative application of the Johnson bound together with $A_q^i(v,k;k)\le \frac{q^v-1}{q^k-1}$ implies 
the anticode bound. The main idea is to consider the subcode consisting of the codewords that all contain a given point $P$, which can also be applied in the setting 
of (full) flag codes:  
\begin{Proposition}
  \label{prop_johnson_point} 
  If $v\ge 2$, then $A_q^f(v,d)\le \gaussm{v}{1}{q}\cdot A_q^f(v-1,d)$.
\end{Proposition}
\begin{proof}
  Let $\mathcal{C}$ be a full flag code in $\mathbb{F}_q^v$ with minimum Grassmann distance $d$. If two different codewords 
  $\Lambda=\left(W_1,\dots,W_{v-1}\right)$ and $\Lambda'=\left(W_1',\dots,W_{v-1}'\right)\in\mathcal{C}$ satisfy $W_1=W_1'=P$ for some point $P\le \mathbb{F}_q^v$,
  then we can write $W_i=\left\langle P,U_{i-1}\right\rangle$ and $W_i'=\left\langle P,U_{i-1}'\right\rangle$ for all $2\le i\le v-1$, 
  where $\dim(U_i)=\dim(U_i')=i$ for all $1\le i\le v-2$. Now, observe that $\fdist(\Lambda,\Lambda')=$ 
  $$
    \sum_{i=1}^{v-1}\left(i-\dim(W_i\cap W_i')\right)=\sum_{i=1}^{v-2}\left(i-\dim(U_i\cap U_i')\right)
    =\fdist\Big(\left(U_1,\dots,U_{v-2}\right),\left(U_1',\dots,U_{v-2}'\right)\Big),
  $$
  so that $\#\left\{\left(W_1,\dots,W_{v-1}\right)\in\mathcal{C}\mid W_1=P\right\}\le A_q^f(v-1,d)$.
\end{proof}

\begin{Corollary}
  \label{cor_a_6_6_j}
   $$
     A_q^f(6,6)\le \gaussm{6}{1}{q}\cdot \left(q^3+1\right)=q^8 + q^7 + q^6 + 2q^5 + 2q^4 + 2q^3 + q^2 + q + 1
   $$
\end{Corollary}
\begin{proof}
  Since $A_q^f(5,6)=q^3+1$, see Proposition~\ref{prop_a_max_distance}, the stated upper bound follows from Proposition~\ref{prop_johnson_point}.
\end{proof}
Note that Corollary~\ref{cor_a_6_6_j} improves upon Proposition~\ref{prop_a_6_6}. Moreover, in all cases where $d$ is small enough, so that $A_q^f(v-1,d)>1$,    
the so far stated upper bounds are indeed implied by Proposition~\ref{prop_johnson_point}. \highlight{We remark, that with respect to upper bounds for 
$A_q^i(v,d;k)$, where $d<k$, almost all of the tightest known upper bounds are given by either the Johnson bound or a slight improvement based on divisible codes, see 
\cite[Theorem 12]{kiermaier2020lengths}. The only two exceptions are given by $(v,d;k)=(6,2;3)$, $(8,3;4)$ in the binary case $q=2$ and obtained via exhaustive integer 
linear programming computations. However, it is not clear if similar techniques may result in strict improvements for full flag codes. For non-full flag codes we refer to 
Footnote~\ref{fn_improved_johnson}.}  

For the cases where the minimum Grassmann distance $d$ is so large that it violates the condition of Proposition~\ref{prop_johnson_point}, we state: 
\begin{Proposition}
  \label{prop_cdc}
  Let $r=\alpha\cdot e_i$ with $\mathbf{0}\le r\le m(v)$ and $d>\sum_{i=1}^{v-1} \left(m(v)_i-\overline{r}_i\right)$, where $\alpha\in\mathbb{N}_{>0}$ and 
  $e_i$ denotes the $i$th unit vector ($1\le i\le v-1$). Then, we have $A_q^f(v,d)\le A_q^i(v,m(v)_i-r_i+1;i)$, where $m(v)_i-r_i+1=\min\{i,v-i\}-r_i+1$.
\end{Proposition}
\begin{proof}
  Let $\mathcal{C}$ be a full flag code in $\mathbb{F}_q^v$ with minimum Grassmann distance $d$. From Corollary~\ref{cor_clique} and Lemma~\ref{lemma_r_bar} we conclude 
  $\idist(W_i,W_i')\le m(v)_i-r_i+1$ for each pair of codewords $\left(W_1,\dots,W_{v-1}\right)$ and $\left(W_1',\dots,W_{v-1}'\right)$ in $\mathcal{C}$, so that
  $$
    \#\mathcal{C}=\#\left\{W_i\mid \left(W_1,\dots,W_{v-1}\right) \in\mathcal{C}\right\}\le A_q^i(v,m(v)_i-r_i+1;i).
  $$
\end{proof}

We can e.g.\ conclude Proposition~\ref{prop_a_max_distance} from Proposition~\ref{prop_cdc}. 

\begin{Corollary}
  \label{cor_cdc}
  For $0\le \delta<\left\lfloor v/2\right\rfloor$ we have 
  $A_q^f(v,d_{max}-\delta)\le A_q^i(v,k;k)$, where $k=\left\lfloor v/2\right\rfloor-\delta$ and $d_{\max}=\sum_{i=1}^{v-1} m(v)_i=\left\lfloor(v/2)^2\right\rfloor$.
\end{Corollary}
\begin{proof}
  Let $\hat{v}=\left\lfloor v/2\right\rfloor$ and $r=e_{\hat{v}}$. We can easily check that $\overline{r}=\sum_{i=\hat{v}-\delta}^{\hat{v}} e_i$, i.e., $\overline{r}$ 
  consists of $\delta+1$ ones. Thus, we can apply Proposition~\ref{prop_cdc}. 
\end{proof}

Based on the recursive application of Proposition~\ref{prop_johnson_point} and Proposition~\ref{prop_cdc} we can state a general explicit upper 
bound for $A_q^f(v,d)$ if we only focus on the leading coefficient:  
\begin{Proposition}
  \label{prop_asymptotic_upper_bound}
  If $d\ge 1$ and $v\ge \left\lceil 2\sqrt{d}\right\rceil\ge 2$, then we have
  $$A_q^f(v,d)\le q^\beta+O\!\left(q^{\beta-1}\right),$$
  where $\beta=\frac{v(v-1)-\hat{v}(\hat{v}-1)}{2}+\hat{v}-d+\left\lfloor(\hat{v}-1)^2/4\right\rfloor$ and $\hat{v}=\left\lceil 2\sqrt{d}\right\rceil$.
\end{Proposition}
\begin{proof}
  Applying Proposition~\ref{prop_johnson_point} $v-\hat{v}\ge 0$ times gives
  $$
    A_q^f(v,d)\le \left(\prod_{i=\hat{v}+1}^{v} \gaussm{i}{1}{q}\right)\cdot A_q^f(\hat{v},d),
  $$  
  so that
  $$
    A_q^f(v,d)\le \left(q^\alpha+O(q^{\alpha-1})\right)\cdot A_q^f(\hat{v},d),   
  $$
  where $\alpha=\frac{v(v-1)-\hat{v}(\hat{v}-1)}{2}$, since $\gaussm{i}{1}{q}=q^{i-1}+O(q^{i-2})$ and $\sum_{i=\hat{v}+1}^v (i-1)=\alpha$. 
  Now note that $\hat{v}=\left\lceil 2\sqrt{d}\right\rceil$ implies the existence of an integer $\delta$ satisfying  
  $0\le \delta<\left\lfloor \hat{v}/2\right\rfloor$ and $d=d_{max}-\delta$, i.e., we can apply Corollary~\ref{cor_cdc}. So, from  
  Corollary~\ref{cor_cdc} we conclude
  $$
    A_q^f(\hat{v},d)\le A_q^i(\hat{v},k;k)\le\gaussm{\hat{v}}{1}{q}/\gaussm{k}{1}{q}\le q^{\hat{v}-k}+O\!\left(q^{\hat{v}-k-1}\right),
  $$
  where $k=d-\left\lfloor(\hat{v}-1)^2/4\right\rfloor$. Thus, we have $A_q^f(v,d)\le q^\beta+O\!\left(q^{\beta-1}\right)$.
\end{proof}
We remark that $A_q(f,d)=1$ if $v< \left\lceil 2\sqrt{d}\right\rceil$, since the maximum possible distance is violated. 
 
In Table~\ref{table_asymptotic_beta} we list the values of $\beta$ in Proposition~\ref{prop_asymptotic_upper_bound} for $v\le 7$. 
\highlight{In Table~\ref{table_asymptotic_sphere_packing} and Table~\ref{table_asymptotic_sphere_covering} we can see that the sphere packing 
and the sphere covering bound, mentioned in Section~\ref{sec_preliminaries}, yield larger exponents in several instances $(v,d)$.} 

\begin{table}[htp]
  \begin{center}
    \begin{tabular}{rrrrrrrrrrrrr}
    \hline
    $v/d$ & 1 & 2 & 3 & 4 & 5 & 6 & 7 & 8 & 9 & 10 & 11 & 12\\
    \hline
    2 &  1 \\
    3 &  3 &  2 \\
    4 &  6 &  5 &  3 &  2 \\
    5 & 10 &  9 &  7 &  6 &  4 &  3 \\
    6 & 15 & 14 & 12 & 11 &  9 &  8 &  5 &  4 & 3 \\ 
    7 & 21 & 20 & 18 & 17 & 15 & 14 & 11 & 10 & 9 & 6 & 5 & 4 \\  
    \hline
    \end{tabular}
    \caption{Values of $\beta$ in Proposition~\ref{prop_asymptotic_upper_bound}, i.e., $A_q^f(v,d)\le q^\beta+O\!\left(q^{\beta-1}\right)$.}
    \label{table_asymptotic_beta}
  \end{center}
\end{table}

\begin{table}[htp]
  \begin{center}
    \begin{tabular}{rrrrrrrrrrrrr}
    \hline
    $v/d$ & 1 & 2 & 3 & 4 & 5 & 6 & 7 & 8 & 9 & 10 & 11 & 12\\
    \hline
    2 &  1 \\
    3 &  3 &  3 \\
    4 &  6 &  6 &  5 &  5 \\
    5 & 10 & 10 &  9 &  9 &  7 &  7 \\
    6 & 15 & 15 & 14 & 14 & 12 & 12 & 10 & 10 &  8 \\ 
    7 & 21 & 21 & 20 & 20 & 18 & 18 & 16 & 16 & 14 & 12 & 12 & 10 \\  
    \hline
    \end{tabular}
    \caption{Exponents $e$ such \highlight{that} the sphere packing bound for $A_q^f(v,d)$ is $\Theta\!\left(q^e\right)$.}
    \label{table_asymptotic_sphere_packing}
  \end{center}
\end{table}

\begin{table}[htp]
  \begin{center}
    \begin{tabular}{rrrrrrrrrrrrr}
    \hline
    $v/d$ & 1 & 2 & 3 & 4 & 5 & 6 & 7 & 8 & 9 & 10 & 11 & 12\\
    \hline
    2 &  1 \\
    3 &  3 &  2 \\
    4 &  6 &  5 &  3 &  1 \\
    5 & 10 &  9 &  7 &  5 &  3 &  2 \\
    6 & 15 & 14 & 12 & 10 &  8 &  6 &  5 & 3 & 1 \\ 
    7 & 21 & 20 & 18 & 16 & 14 & 12 & 10 & 9 & 7 & 5 & 3 & 2 \\  
    \hline
    \end{tabular}
    \caption{Exponents $e$ such \highlight{that} the sphere covering bound for $A_q^f(v,d)$ is $\Theta\!\left(q^e\right)$.}
    \label{table_asymptotic_sphere_covering}
  \end{center}
\end{table}

\section{Bounds for non-full flags and other variants}
\label{sec_non_full}

In this section we want to merely consider a few examples in order to shed some light on the general picture.

\begin{Example}
  \label{ex_a_2_f_6_5_234}
  We can easily generalize Definition~\ref{def_r_bar} and Theorem~\ref{thm_anticode_bound} to the situation $T\subsetneq \{1,\dots,v-1\}$. For a flag code $\mathcal{C}$ 
  in $\mathbb{F}_2^6$ of type $T=\{2,3,4\}$ \highlight{and minimum distance at least $5$} we obtain :
  \begin{itemize} 
    \item $\overline{(2,0,0)}=(2,2,0)$ $\leadsto$ $\#\mathcal{C}\le\gaussm{6}{2}{2}=651$;
    \item $\overline{(1,2,0)}=(1,2,0)$ $\leadsto$ $\#\mathcal{C}\le\frac{\gaussm{6}{1}{2}\cdot\gaussm{5}{1}{2}}{\gaussm{2}{1}{2}\cdot\gaussm{2}{1}{2}}=217$;
    \item $\overline{(0,3,0)}=(1,3,1)$ $\leadsto$ $\#\mathcal{C}\le\gaussm{6}{3}{2}=1395$;
    \item $\overline{(1,0,1)}=(1,1,1)$ $\leadsto$ $\#\mathcal{C}\le\frac{\gaussm{6}{1}{2}\cdot\gaussm{5}{2}{2}}{\gaussm{2}{1}{2}\cdot\gaussm{3}{1}{2}}=465$,
  \end{itemize}
  so that $A_2^f(6,5;\{2,3,4\})\le 217$.   
\end{Example}

The vector $r=(1,0,1)$ with $\overline{r}=(1,1,1)$ is of special interest with respect to the relation of Lemma~\ref{lemma_clique} and Corollary~\ref{cor_clique}, where 
the latter is the one used in \highlight{Theorem~\ref{thm_anticode_bound}.} 
Going along Corollary~\ref{cor_clique} we would consider 
the flag of a point $P$ and a plane $E$ with $P\le E$ such that there is at most one codeword $\left(W_2,W_3,W_4\right)$ with $P\le W_2$ and $E\le W_4$. Using 
the more general Lemma~\ref{lemma_clique}, we can also consider the flag of a point $P$ and a $5$-space $K$ with $P\le K$ to conclude that there is at most one codeword 
$\left(W_2,W_3,W_4\right)$ with $P\le W_2$ and $W_4\le K$. There are $\gaussm{6}{1}{2}\cdot\gaussm{5}{1}{2}$ such flags $P\le K$ in total and for each fixed 
codeword $\left(W_2,W_3,W_4\right)$ there are $\gaussm{2}{1}{2}\cdot\gaussm{2}{1}{2}$ flags $P\le K$ with $P\le W_2$ and $W_4\le K$. Thus, $A_2^f(6,5;\{2,3,4\})\le 
\frac{63\cdot 31}{3\cdot 3}=217$.    

The underlying idea of Proposition~\ref{prop_johnson_point} can also be generalized easily, i.e., if $\mathcal{C}$ is a flag code in $\mathbb{F}_2^6$ of 
type $T=\{2,3,4\}$ and minimum Grassmann distance $d=5$, then given a point $P$ the set
$$
  \mathcal{C}_P:=\left\{\left(W_2,W_3,W_4\right)\in\mathcal{C}\mid P\le W_2\right\}
$$ 
corresponds to a flag code in $\mathbb{F}_2^5$ of type $\{1,2,3\}$ and minimum Grassmann distance $d=5$. Thus, $\#\mathcal{C}_P\le A_2^f(5,5;\{1,2,3\})$ 
and $A_2^f(6,5;\{2,3,4\})\le \frac{\gaussm{6}{1}{2}}{\gaussm{2}{1}{2}}\cdot A_2^f(5,5;\{1,2,3\})$. For $A_2^f(5,5;\{1,2,3\})$ we observe that the $2$-spaces 
in the middle layer of the codewords have to give a partial line spread in $\mathbb{F}_2^5$, so that $A_2^f(5,5;\{1,2,3\})\le 9$ and 
$A_2^f(6,5;\{2,3,4\})\le\frac{63}{3}\cdot 9=189$, which improves upon the previously stated upper bounds.    

\begin{Example}
  Let us consider some upper bounds for $A_2^f(7,3;\{3,4\})$\highlight{:}
  \begin{itemize} 
    \item $\overline{(3,0)}=(3,2)$ $\leadsto$ $A_2^f(7,3;\{3,4\})\le \highlight{\gaussm{7}{3}{2}}=11811$;
    \item $\overline{(2,2)}=(2,2)$ $\leadsto$ $A_2^f(7,3;\{3,4\})\le\frac{\gaussm{7}{2}{2}\cdot\gaussm{5}{1}{2}}{\gaussm{3}{2}{2}\cdot\gaussm{2}{1}{2}}=3937$;
    \item $\overline{(0,3)}=(3,3,)$ $\leadsto$ $A_2^f(7,3;\{3,4\})\le \highlight{\gaussm{7}{3}{2}}=11811$.
  \end{itemize} 
  Alternatively, by considering all codewords $\left(W_3,W_4\right)$, where $W_3$ contains a fixed point $P$, cf.\ Proposition~\ref{prop_johnson_point}, 
  we obtain
  \begin{equation}
    \label{ie_special_734}
    A_2^f(7,3;\{3,4\})\le \highlight{\frac{\gaussm{7}{1}{2}}{\gaussm{3}{1}{2}}\cdot A_2^f(6,3;\{2,3\})}.
  \end{equation}
  For $A_2^f(6,3;\{2,3\})$ we can use the argument again and obtain $A_2^f(6,3;\{2,3\})\le \frac{63}{3}\cdot A_2^f(5,3;\{1,2\})$. Since the lines of the second layer 
  of the codewords have to give a partial line spread in $\mathbb{F}_2^5$, we have $A_2^f(5,3;\{1,2\})\le 9$ (indeed, we have $A_2^f(5,3;\{1,2\})= 9$), so that 
  $A_2^f(6,3;\{2,3\})\le 189$. Thus, Inequality~(\ref{ie_special_734}) yields $A_2^f(7,3;\{3,4\})\le 3429$, which improves upon the previously stated upper bounds. 
  \highlight{In the context of non-full flag codes the improvement of the Johnson bound for constant dimension codes, see \cite[Theorem 12]{kiermaier2020lengths}, 
  might be adjusted and applied successfully.}\footnote{\label{fn_improved_johnson}Let us assume $A_2^f(6,3;\{2,3\})\le 185$ for a moment. Inequality~(\ref{ie_special_734}) 
  then would yield $A_2^f(7,3;\{3,4\})\le \frac{127}{7}\cdot 185=3356 +\frac{3}{7}$. 
Of course this can be rounded down to $3356$, since $A_2^f(7,3;\{3,4\})$ is an integer. However, as in the case of constant dimension codes the rounding of the 
Johnson bound can be improved using the theory of $q^r$-divisible codes, see \cite{kiermaier2020lengths}. More concretely, for each codeword $\left(W_3,W_4\right)$ we 
just consider the plane $W_3$.  Since we assume $A_2^f(6,3;\{2,3\})\le 185$ those planes cover each point of $\mathbb{F}_2^7$ at most $185$ times. If the flag code 
has cardinality $3356$ then not every point of $\mathbb{F}_2^7$ can be covered exactly $185$ times, i.e., the missing points correspond to a multiset of points 
of cardinality $3$, which in turn corresponds to a binary linear code of effective length $3$. Since it can be shown that this code has to be $4$-divisible, i.e., 
the weight of every codeword has to be divisible by $4$ and such a code cannot exist, we could strengthen our argument to 
$A_2^f(7,3;\{3,4\})\le 3355$. (A $4$-divisible binary linear code of effective length $10$ indeed exists.) For the details we refer to \cite[Lemma 13(i)]{kiermaier2020lengths} 
and its preparing results and definitions.    
}      
\end{Example}

Another variant is to consider sets of elements of the Cartesian product $\gauss{v}{1}{q}\times \gauss{v}{2}{q}\times\dots\times \gauss{v}{v-1}{q}$ as codes with respect 
to the Grassman distance. By $A_q^c(v,d)$ we denote the corresponding maximum cardinality of such a code with minimum Grassmann distance $d$. Obviously we have 
$A_q^f(v,d)\le A_q^c(v,d)$ and $d\le \left\lfloor(v/2)^2\right\rfloor$. If we replace $\bar{r}$ by $r$ then the modified version of Theorem~\ref{thm_anticode_bound} 
holds for $A_q^c(v,d)$. As a special case we obtain the same upper for the maximum possible Grassmann distance for $A_q^c(v,\left\lfloor(v/2)^2\right\rfloor)$ as for 
$A_q^f(v,\left\lfloor(v/2)^2\right\rfloor)$, so that:
\begin{Proposition}
  \label{prop_a_max_distance_cardesian}
  For each integer $k\ge 1$ we have
  $$
    A_q^c(2k,k^2)=q^k+1  
  $$
  and for each integer $k\ge 2$ we have
  $$
    A_q^c(2k+1,k^2+k)=q^{k+1}+1.  
  $$
\end{Proposition}
\highlight{So, we have $A_q^c(2k,k^2)=A_q^f(2k,k^2)$ and $A_q^c(2k+1,k^2+k)=A_q^f(2k+1,k^2+k)$. Later we will see that $A_q^c(v,d)=A_q^f(v,d)$ is not true in general.}  

Similar as for flag codes we can restrict the possible dimensions of the parts of a codeword to a subset $\emptyset\neq T\subseteq\{1,\dots,v-1\}$, which we call type. 
More precisely, codewords are elements of \highlight{the Cartesian product} $\bigtimes_{t\in T} \gauss{v}{t}{q}$. By $A_q^c(v,d;T)$ we denote the corresponding maximum 
possible cardinality of such a code. For Grassmann distance $d=1$ we have
\begin{equation}
  A_q^c(v,1;T)=\prod_{t\in T} \gaussm{v}{t}{q}   
\end{equation}
and
\begin{equation}
  A_q^c(v,1)=\prod_{t=1}^{v-1} \gaussm{v}{t}{q}.   
\end{equation}

In order to show that $A_q^f(v,d;T)$ and $A_q^c(v,d;T)$ can have different orders of magnitude in terms of the field size $q$ we consider the example $(v,d)=(5,2)$ and 
$T=\{2,3\}$. Since $\fdist\big((L,E),(L,E'))\le 1$ for a line $L$ and two planes $E,E'$ containing $L$, we have
$$
  A_q^f(5,2;\{2,3\})\le \gaussm{5}{2}{q}=q^6+q^5+2q^4+2q^3+2q^2+q+1.
$$  
Next we want to construct a larger lower bound for $A_q^c(5,2;\{2,3\})$ and introduce some necessary notation. For two matrices $A,B\in\mathbb{F}_q^{m\times n}$ we define 
the rank distance $\rdist(A,B):=\rk(A-B)$. A subset $\mathcal{M}\subseteq \mathbb{F}_q^{m\times n}$ is called a rank metric code. 
\begin{Theorem}(see \cite{gabidulin1985theory})
  \label{thm_MRD_size}
  Let $m,n\ge d'$ be positive integers, $q$ a prime power, and $\mathcal{M}\subseteq \mathbb{F}_q^{m\times n}$ be a rank metric
  code with minimum rank distance $d'$. Then, $\# \mathcal{M}\le q^{\max\{n,m\}\cdot (\min\{n,m\}-d'+1)}$.
\end{Theorem}
Codes attaining this upper bound are called maximum rank distance (\MRD) codes. They exist for all choices of parameters, which remains true if we restrict 
to linear rank metric codes, see  \cite{gabidulin1985theory}. For e.g.\ $(m,n)=(2,3)$ and $d'=2$ there exists an {\MRD} code $\mathcal{M}^{2,3}$ of cardinality 
$q^3$. For a general $m\times n$ {\MRD} code $\mathcal{M}$ we can associate to each matrix $M\in\mathcal{M}$ the rowspace $\langle(I_{m\times m}|M)\rangle$ of the 
concatenation of the $m\times m$ unit matrix $I_{m\times m}$ and matrix $M$, which is an $m$-dimensional subspace of $\mathbb{F}_q^{m+n}$. The construction of a subspace 
from a matrix is also called \emph{lifting}. If $U=\langle(I_{m\times m}|M)\rangle$ and $W=\langle(I_{m\times m}|M')\rangle$ are two subspaces 
lifted from two matrices, then $\idist(U,W)=\rdist(M,M')$. Thus, $\mathcal{M}^{2,3}$ can be lifted to a set of $q^3$ lines in $\mathbb{F}_q^5$ with pairwise 
injection distance $2$, i.e., a partial line spread. Since $\mathcal{M}^{2,3}$ is linear the $q^6$ $2\times 3$ matrices over $\mathbb{F}_q$ can be partitioned into 
$q^3$ $2\times 3$ {\MRD} codes with minimum rank distance $2$. By lifting we obtain $q^6$ lines $U'_{i,j}$ in $\mathbb{F}_q^5$, where $1\le i\le q^3$ and $1\le j\le q^3$, such that
\begin{eqnarray*}
  \idist\!\left(U'_{i,j},U'_{i',j'}\right)=2 &\text{if}& i\neq i', j=j',\\
  \idist\!\left(U'_{i,j},U'_{i',j'}\right)=1 &\text{if}& j\neq j',\text{ and}\\
  \idist\!\left(U'_{i,j},U'_{i',j'}\right)=0 &\text{if}& i=i', j=j'.
\end{eqnarray*}     
By \highlight{duplicating} this configuration $q^3$ times we obtain $q^3$ lines $U_{i,j,h}$, where $1\le i,j,h\le q^3$, such that  
\begin{eqnarray*}
  \idist\!\left(U_{i,j,h},U_{i',j',h'}\right)=2 &\text{if}& i\neq i', j=j',\\
  \idist\!\left(U_{i,j,h},U_{i',j',h'}\right)=1 &\text{if}& j\neq j',\text{ and}\\
  \idist\!\left(U_{i,j,h},U_{i',j',h'}\right)=0 &\text{if}& i=i', j=j'.
\end{eqnarray*}  
Starting from a $3\times 2$ {\MRD} code $\mathcal{M}^{3\times 2}$ with minimum rank distance $2$ and cardinality $q^3$ we can similarly construct $q^9$ planes 
$W_{i,j,h}$, where $1\le i,j,h\le q^3$ such that 
\begin{eqnarray*}
  \idist\!\left(W_{i,j,h},W_{i',j',h'}\right)=2 &\text{if}& h\neq h', j=j',\\
  \idist\!\left(W_{i,j,h},W_{i',j',h'}\right)=1 &\text{if}& j\neq j',\text{ and}\\
  \idist\!\left(W_{i,j,h},W_{i',j',h'}\right)=0 &\text{if}& h=h', j=j'.
\end{eqnarray*}    
With this we can construct a flag code $\mathcal{C}=\left\{\left(U_{i,j,h},W_{i,j,h}\right)\,:\,1\le i,j,h\le q^3\right\}$ of type $\{2,3\}$ and cardinality $q^9$. 
It can be easily checked that 
$$
  \fdist\Big(\left(U_{i,j,h},W_{i,j,h}\right),\left(U_{i',j',h'},W_{i',j','h}\right)\Big)=
  \idist\Big(U_{i,j,h},U_{i',j',h'}\Big)+\idist\Big(W_{i,j,h},W_{i',j',h'}\Big)\ge 2
$$
if $(i,j,h)\neq (i',j',h')$. Thus, the minimum Grassman distance of $\mathcal{C}$ is at least $2$ and $A_q^c(5,2;\{2,3\})\ge q^9$. By considering the codewords $(L,E)$ 
with a fixed line $L$ and planes $E$ contained in a hyperplane $H$ of $\mathbb{F}_q^5$ we may show $A_q^c(5,2;\{2,3\})\le \gaussm{5}{2}{q}\cdot\gaussm{5}{1}{q}=q^{10}+
O\!\left(q^9\right)$.

We can also extend our construction to a full \highlight{Cartesian product code} for $(v,d)=(5,3)$. To this end let $U_{i,j}^2$ be lines in $\mathbb{F}_q^5$ and $U_{i,j}^3$ be planes in 
$\mathbb{F}_q^5$ for $1\le i,j\le q^4$ such that 
$$
\fdist\Big(\left(U_{i,j}^2,U_{i,j}^3\right),\left(U_{i',j'}^2,U_{i',j'}^3\right)\Big) \ge 2
$$ 
whenever $(i,j)\neq (i',j')$. Since there are $\gaussm{5}{1}{q}\ge 4$ points in $\mathbb{F}_q^5$, we can choose $q^8$ points $U_{i,j}^1$ such that 
$\idist\Big(U_{i,j}^1,U_{i',j'}^1\Big)=1$ if $j\neq j'$ and zero otherwise. Similarly, we can choose $q^8$ \highlight{hyperplanes} $U_{i,j}^4$ such that 
$\idist\Big(U_{i,j}^4,U_{i',j'}^1\Big)=4$ if $i\neq i'$ and zero otherwise. With this we can check that
$$
  \mathcal{C}=\left\{\Big(U_{i,j}^1,U_{i,j}^2,U_{i,j}^3,U_{i,j}^4\Big)\,:\,1\le i,j\le q^4\right\}
$$
is a full flag code in $\mathbb{F}_q^5$ with cardinality $q^8$ and minimum Grassmann distance $3$. Thus, we have $A_q^c(5,3)\ge q^8$, while $A_q^f(5,3)\le q^7+
O\!\left(q^6\right)$.

\section{Exact values and bounds for small parameters}
\label{sec_exact_values_and_bounds}

In this section we summarize the exact values and bounds for $A_q^f(v,d)$, \highlight{where $v\le 6$}, from the previous sections. We start with the known exact formulas that are parametric in $q$   
from Section~\ref{sec_preliminaries}, i.e., for $d=1$, $d=\left\lfloor(v/2)^2\right\rfloor$ 
(Proposition~\ref{prop_a_max_distance}), $(v,d)=(3,2)$ (Propositions~\ref{prop_a_3_2}), and  $(v,d)=(4,3)$ (Propositions~\ref{prop_a_4_3}). 

 \begin{equation}
  A_q^f(2,1)=q+1
\end{equation}

\begin{equation}
  A_q^f(3,1)=\left(q+1\right)\cdot\left(q^2+q+1\right)=q^3 + 2q^2 + 2q + 1
\end{equation}

\begin{equation}
  A_q^f(3,2)=q^2 + q+1
\end{equation}

\begin{equation}
  A_q^f(4,1)=\left(q+1\right)\cdot\left(q^2+q+1\right)\cdot\left(q^3+q^2+q+1\right)=q^6 + 3q^5 + 5q^4 + 6q^3 + 5q^2 + 3q + 1
\end{equation}

\begin{equation}
  A_q^f(4,3)=q^3+q^2 + q+1
\end{equation}

\begin{equation}
  A_q^f(4,4)=q^2 + 1
\end{equation}

\begin{eqnarray}
  A_q^f(5,1)&=&\left(q+1\right)\cdot\left(q^2+q+1\right)\cdot\left(q^3+q^2+q+1\right)\cdot\left(q^4+q^3+q^2+q+1\right)\\ 
  &=&q^{10} + 4q^9 + 9q^8 + 15q^7 + 20q^6 + 22q^5 + 20q^4 + 15q^3 + 9q^2 + 4q + 1\notag
\end{eqnarray}

\begin{equation}
  A_q^f(5,6)=q^3 + 1
\end{equation}

\begin{eqnarray}
  A_q^f(6,1)&=&\left(q\!+\!1\right)\left(q^2\!+\!q\!+\!1\right)\left(q^3\!+\!q^2\!+\!q\!+\!1\right)\left(
  q^4\!+\!q^3\!+\!q^2\!+\!q\!+\!1\right)\left(q^5\!+\!q^4\!+\!q^3\!+\!q^2\!+\!q\!+\!1\right)\\ 
  &=& q^{15} + 5q^{14} + 14q^{13} + 29q^{12} + 49q^{11} + 71q^{10} + 90q^9 + 101q^8 + 101q^7 + 90q^6\notag\\ 
  && + 71q^5 + 49q^4 + 29q^3 + 14q^2 + 5q + 1\notag
\end{eqnarray}

\begin{equation}
  A_q^f(6,9)=q^3 + 1
\end{equation}

We continue with parametric upper bounds. Propositions~\ref{prop_a_4_2}, \ref{prop_a_6_7}, \ref{prop_a_6_8} and Corollary~\ref{cor_a_6_6_j} state
\begin{eqnarray}
  A_q^f(4,2)&\le& q^5 + 2q^4 + 3q^3 + 3q^2 + 2q + 1,\\ 
  A_q^f(6,6)&\le& q^8 + q^7 + q^6 + 2q^5 + 2q^4 + 2q^3 + q^2 + q + 1, \\
  A_q^f(6,7)&\le& q^5+q^4+q^3+q^2+q+1,\text{ and}\\
  A_q^f(6,8)&\le& q^4+q^2+1.
\end{eqnarray}

Next we complete the missing parametric cases $(v,d)$ for $v\le 5$. To this end we use clique constraints corresponding to $\mathcal{V}^{r}_{v,q}$ 
for a suitable reduction vector $r$, i.e., we apply \highlight{Theorem~\ref{thm_anticode_bound}} 
to evaluate the involved cardinalities.

\begin{Proposition}
  \label{prop_a_5_2}
  \begin{equation}
    A_q^f(5,2)\le\gaussm{5}{1}{q}\cdot \gaussm{4}{1}{q}\cdot \gaussm{3}{1}{q}=q^9 + 3q^8 + 6q^7 + 9q^6 + 11q^5 + 11q^4 + 9q^3 + 6q^2 + 3q + 1
  \end{equation}  
\end{Proposition}
\begin{proof}
  Since $\overline{(1,2,2,0)}=(1,2,2,1)$ the stated upper bound is obtained from the clique constraints corresponding to $\mathcal{V}^{1,2,2,0}_{5,q}$, 
  \highlight{i.e.\ we apply Theorem~\ref{thm_anticode_bound}.}
\end{proof}
For $q=2$ prescribing a Singer cycle, i.e., a cyclic group of order $31$, the ILP from Section~\ref{sec_ILP} has an optimal target value of $3069$, while the upper bound of 
Proposition~\ref{prop_a_5_2} yields $A_2^f(5,2)\le 3255$. 

\begin{Proposition}
  \label{prop_a_5_3}
  \begin{equation}
    A_q^f(5,3)\le\gaussm{5}{1}{q}\cdot \gaussm{4}{1}{q}=q^7 + 2q^6 + 3q^5 + 4q^4 + 4q^3 + 3q^2 + 2q + 1
  \end{equation}  
\end{Proposition}
\begin{proof}
  Since $\overline{(1,2,0,0)}=(1,2,1,0)$ the stated upper bound is obtained from the clique constraints corresponding to $\mathcal{V}^{1,2,0,0}_{5,q}$, 
  \highlight{i.e.\ we apply Theorem~\ref{thm_anticode_bound}.}
\end{proof}
We remark that Proposition~\ref{prop_a_5_3} is tight for $q=2$, i.e., a corresponding code of cardinality $465$ indeed exists. Such a code also exists if we prescribe 
a Singer cycle, i.e., a cyclic group of order $31$. 

\begin{Proposition}
  \label{prop_a_5_4}
  \begin{equation}
    A_q^f(5,4)\le\gaussm{5}{1}{q}\cdot \left(q^2+1\right)=q^6 + q^5 + 2q^4 + 2q^3 + 2q^2 + q + 1
  \end{equation}
\end{Proposition}
\begin{proof}
  Since $\overline{(1,0,1,0)}=(1,1,1,0)$ the stated upper bound is obtained from the clique constraints corresponding to $\mathcal{V}^{1,0,1,0}_{5,q}$, 
  \highlight{i.e.\ we apply Theorem~\ref{thm_anticode_bound}.} 
\end{proof}
We remark that Proposition~\ref{prop_a_5_4} is tight for $q=2$, i.e., a corresponding code of cardinality $155$ indeed exists. Such a code also exists if we prescribe 
a Singer cycle, i.e., a cyclic group of order $31$ \highlight{generated by} 
$$
  \highlight{\begin{pmatrix}
    0 & 1 & 0 & 0 & 0\\
    0 & 0 & 1 & 0 & 0\\
    0 & 0 & 0 & 1 & 0\\
    0 & 0 & 0 & 0 & 1\\
    1 & 0 & 1 & 1 & 1
  \end{pmatrix}.}
$$
\highlight{The flag code is given by five orbits of size $31$ and corresponding representatives are given by:}
\highlight{ 
$$
\left(
\left\langle\begin{pmatrix}
0\,0\,0\,0\,1 
\end{pmatrix}\right\rangle,
\left\langle\begin{pmatrix}
0\,0\,0\,1\,0\\ 
0\,0\,0\,0\,1 
\end{pmatrix}\right\rangle,
\left\langle\begin{pmatrix}
1\,1\,0\,0\,0\\ 
0\,0\,0\,1\,0\\ 
0\,0\,0\,0\,1 
\end{pmatrix}\right\rangle,
\left\langle\begin{pmatrix}
1\,1\,0\,0\,0\\ 
0\,0\,1\,0\,0\\ 
0\,0\,0\,1\,0\\ 
0\,0\,0\,0\,1
\end{pmatrix}\right\rangle
\right) 
$$
$$ 
\left(
\left\langle\begin{pmatrix}
0\,0\,0\,0\,1 
\end{pmatrix}\right\rangle,
\left\langle\begin{pmatrix}
0\,0\,1\,0\,0\\ 
0\,0\,0\,0\,1 
\end{pmatrix}\right\rangle,
\left\langle\begin{pmatrix}
1\,0\,0\,1\,0\\ 
0\,0\,1\,0\,0\\ 
0\,0\,0\,0\,1 
\end{pmatrix}\right\rangle,
\left\langle\begin{pmatrix}
1\,0\,0\,0\,0\\ 
0\,0\,1\,0\,0\\ 
0\,0\,0\,1\,0\\ 
0\,0\,0\,0\,1 
\end{pmatrix}\right\rangle
\right)
$$
$$
\left(
\left\langle\begin{pmatrix}
0\,0\,0\,0\,1 
\end{pmatrix}\right\rangle,
\left\langle\begin{pmatrix}
0\,1\,0\,1\,0\\ 
0\,0\,0\,0\,1 
\end{pmatrix}\right\rangle,
\left\langle\begin{pmatrix}
1\,0\,1\,0\,0\\ 
0\,1\,0\,1\,0\\ 
0\,0\,0\,0\,1 
\end{pmatrix}\right\rangle,
\left\langle\begin{pmatrix}
1\,0\,0\,0\,0\\ 
0\,1\,0\,1\,0\\ 
0\,0\,1\,0\,0\\ 
0\,0\,0\,0\,1 
\end{pmatrix}\right\rangle
\right)
$$
$$
\left(
\left\langle\begin{pmatrix}
0\,0\,0\,0\,1 
\end{pmatrix}\right\rangle,
\left\langle\begin{pmatrix}
0\,1\,1\,1\,0\\ 
0\,0\,0\,0\,1 
\end{pmatrix}\right\rangle,
\left\langle\begin{pmatrix}
0\,1\,0\,0\,0\\ 
0\,0\,1\,1\,0\\ 
0\,0\,0\,0\,1 
\end{pmatrix}\right\rangle,
\left\langle\begin{pmatrix}
0\,1\,0\,0\,0\\ 
0\,0\,1\,0\,0\\ 
0\,0\,0\,1\,0\\ 
0\,0\,0\,0\,1 
\end{pmatrix}\right\rangle
\right)
$$
$$
\left(
\left\langle\begin{pmatrix}
0\,0\,0\,0\,1 
\end{pmatrix}\right\rangle,
\left\langle\begin{pmatrix}
1\,0\,0\,0\,0\\ 
0\,0\,0\,0\,1 
\end{pmatrix}\right\rangle,
\left\langle\begin{pmatrix}
1\,0\,0\,0\,0\\ 
0\,1\,1\,0\,0\\ 
0\,0\,0\,0\,1 
\end{pmatrix}\right\rangle,
\left\langle\begin{pmatrix}
1\,0\,0\,0\,0\\ 
0\,1\,0\,0\,0\\ 
0\,0\,1\,0\,0\\ 
0\,0\,0\,0\,1 
\end{pmatrix}\right\rangle
\right)
$$
}

\begin{Proposition}
  \label{prop_a_5_5}
  \begin{equation}
    A_q^f(5,5)\le\gaussm{5}{1}{q}=q^4+q^3+q^2+q+1
  \end{equation}  
\end{Proposition}
\begin{proof}
  Since $\overline{(1,0,0,0)}=(1,1,0,0)$ the stated upper bound is obtained from the clique constraints corresponding to $\mathcal{V}^{1,0,0,0}_{5,q}$, 
  \highlight{i.e.\ we apply Theorem~\ref{thm_anticode_bound}.}
\end{proof}   
We remark that Proposition~\ref{prop_a_5_5} is tight for $q=2$, i.e., a corresponding code of cardinality $31$ indeed exists. Such a code also exists if we prescribe 
a Singer cycle, i.e., a cyclic group of order $31$. 

\begin{table}[htp]
  \begin{center}
    \begin{tabular}{rrrrrrrrrr}
    \hline
    $v/d$ & 1  & 2          & 3       & 4      & 5      & 6     & 7 & 8 & 9 \\
    \hline
    2 & 3      &            &         &        &        &   \\
    3 & 21     & 7          &         &        &        &   \\ 
    4 & 315    & 105        & 15      & 5      &        &   \\
    5 & 9765   & 3120--3255 & 465     & 155    & 31     & 9 \\
    \hline
    \end{tabular}
    \caption{Bounds \highlight{and exact values} for $A_2^f(v,d)$ for $v\le \highlight{5}$.}
    \label{table_bounds_binary}
  \end{center}
\end{table}

For the binary case $q=2$ we can say a bit more. Except for $(v,d)=(5,2)$ the upper bounds for $v\le 5$ are attained, see Table~\ref{table_bounds_binary}. The lower bounds 
have been mainly obtained using the ILP approach, with prescribed automorphisms, see Section~\ref{sec_ILP} and Section~\ref{sec_bounds} for the details on the chosen groups. 
For $v=6$ and $v=7$ we list the upper bounds resulting from Proposition~\ref{prop_johnson_point} and Proposition~\ref{prop_cdc} in Table~\ref{table_upper_bounds_binary_v_7}. 
As exact values we have $A_2^f(6,1)=615195$, $A_2^f(6,7)=63$,  $A_2^f(6,8)=21$, and $A_2^f(6,\highlight{9})=9$ for $v=6$. The inequalities $224\le A_2^f(6,6) \le 567$ show that it might 
be hard to obtain narrow bounds for $A_2^g(v,d)$ even for medium sized parameters.

\begin{table}[htp]
  \tabcolsep=3.2pt
  \begin{center}
    \begin{tabular}{rrrrrrrrrrrrr}
    \hline
    $v/d$ & 1  & 2          & 3       & 4      & 5      & 6     & 7 & 8 & 9 & 10 & 11 & 12\\
    \hline
    6 & \textbf{615195} & 205065   & 29295 & 9765 & 1953 & 567 & \textbf{63} & \textbf{21} & \textbf{9} \\
    7 & \textbf{78129765} & 26043255 & 3720465 & 1240155 & 248031 & 72009 & 8001 & 2667 & 1143 & 127 & 41 & \textbf{17}\\
    \hline
    \end{tabular}
    \caption{Upper bounds for $A_2^f(6,d)$ and $A_2^f(7,d)$ \highlight{(tight bounds in bold)}.}
    \label{table_upper_bounds_binary_v_7}
  \end{center}
\end{table}

\section{Conclusion and future research}
\label{sec_conclusion}

Comparing the data of Table~\ref{table_asymptotic_sphere_packing} and Table~\ref{table_asymptotic_beta} we conjecture that the upper bounds for $A_q^f(v,d)$ induced by 
Proposition~\ref{prop_johnson_point} and Corollary~\ref{cor_cdc} are always tighter than the \emph{sphere packing bound} for flag codes, see 
\cite{with_flags,liebhold2018generalizing}. Of course it would be interesting to determine an explicit formula for the leading coefficient of the  
sphere packing bound for $A_q^f(v,d)$, or the sphere covering bound, as we have determined for our upper bound in Proposition~\ref{prop_asymptotic_upper_bound}. Intended more 
as an inspiring challenge instead of being based on rigorous insights, we conjecture that the bound of Proposition~\ref{prop_asymptotic_upper_bound} is tight up to the terms of 
lower order. To this end a series of general constructions is desirable, see e.g.~\cite{phd_liebhold,with_flags}, where the authors have shown that flag codes 
can be superior to constant dimension codes. In those cases that we have investigated the order of magnitude of the sphere covering bound is not exceeded.   
Is seems that the parametric construction of good flag codes is a teaser. In those cases in 
Section~\ref{sec_bounds} where proposed upper bounds for $A_q^f(v,d)$ are attained by a flag code with a Singer group as subgroup of automorphisms for $q=2$, we conjecture that 
this is the case for all field sizes $q$.  

One may introduce a more general version of \highlight{Theorem~\ref{thm_anticode_bound}} 
based on Lemma~\ref{lemma_clique} instead of Corollary~\ref{cor_clique}, see the discussion after Example~\ref{ex_a_2_f_6_5_234}. However, it is not clear if the corresponding bounds will be competitive. 

The determination of tighter bounds for $A_2^f(6,d)$ and $A_2^f(7,d)$ seems to be an interesting an challenging open problem. Of course the situation for $A_q^f(v,d;T)$, i.e.\ 
non-full flag codes, and for $A_q^s(v,d)$ is even wider open than it is for $A_q^f(v,d)$.  

\section*{Acknowledgements} 
The author thanks Gabriele Nebe for her comments and remarks on an earlier draft. Especially, the idea to study the quantity $A_q^c(v,d)$ and compare it with $A_q^f(v,d)$ was hers. 
\highlight{Moreover I am indebted to the anonymous reviewers whose remarks and comments allowed me to improve the presentation of the paper.} 


\begin{thebibliography}{10}

\bibitem{alonso2020flag}
C.~Alonso-Gonz{\'a}lez, M.~{\'A}. Navarro-P{\'e}rez, and X.~Soler-Escriv{\`a}.
\newblock Flag codes from planar spreads in network coding.
\newblock {\em \highlight{Finite Fields and Their Applications}},
  \highlight{68}:101745, 2020.

\bibitem{beutelspacher1975partial}
A.~Beutelspacher.
\newblock Partial spreads in finite projective spaces and partial designs.
\newblock {\em Mathematische Zeitschrift}, 145(3):211--229, 1975.

\bibitem{cai2019network}
H.~Cai, T.~Etzion, M.~Schwartz, and A.~Wachter-Zeh.
\newblock Network coding solutions for the combination network and its
  subgraphs.
\newblock In {\em 2019 IEEE International Symposium on Information Theory
  (ISIT)}, pages 862--866. IEEE, 2019.

\bibitem{drudge2002orbits}
K.~Drudge.
\newblock On the orbits of singer groups and their subgroups.
\newblock {\em \highlight{The Electronic Journal of Combinatorics}}, pages
  R15--R15, 2002.

\bibitem{etzion2020subspace}
T.~Etzion, S.~Kurz, K.~Otal, and F.~{\"O}zbudak.
\newblock Subspace packings: constructions and bounds.
\newblock {\em Designs, Codes and Cryptography},
  \highlight{88}:\highlight{1781--1810}, 2020.

\bibitem{fourier2020degenerate}
G.~Fourier and G.~Nebe.
\newblock Degenerate flag varieties in network coding.
\newblock {\em arXiv preprint 2003.02002}, 2020.

\bibitem{gabidulin1985theory}
E.~Gabidulin.
\newblock Theory of codes with maximum rank distance.
\newblock {\em Problemy Peredachi Informatsii}, 21(1):3--16, 1985.

\bibitem{glynn1988set}
D.~G. Glynn.
\newblock On a set of lines of $\operatorname{PG}(3, q)$ corresponding to a
  maximal cap contained in the {K}lein quadric of $\operatorname{PG}(5, q)$.
\newblock {\em Geometriae Dedicata}, 26(3):273--280, 1988.

\bibitem{heinlein2016tables}
D.~Heinlein, M.~Kiermaier, S.~Kurz, and A.~Wassermann.
\newblock Tables of subspace codes.
\newblock {\em arXiv preprint 1601.02864}, 2016.

\bibitem{kiermaier2020lengths}
M.~Kiermaier and S.~Kurz.
\newblock On the lengths of divisible codes.
\newblock {\em IEEE Transactions on Information Theory}, 66(7):4051--4060,
  2020.

\bibitem{kohnert2008construction}
A.~Kohnert and S.~Kurz.
\newblock Construction of large constant dimension codes with a prescribed
  minimum distance.
\newblock In {\em Mathematical methods in computer science}, pages 31--42.
  Springer, 2008.

\bibitem{phd_liebhold}
D.~Liebhold.
\newblock {\em Flag codes with application to network coding}.
\newblock PhD thesis, RWTH Aachen, 2019.

\bibitem{with_flags}
D.~Liebhold, G.~Nebe, and A.~Vazquez-Castro.
\newblock Network coding with flags.
\newblock {\em Designs, Codes and Cryptography}, 86(2):269--284, 2018.

\bibitem{liebhold2018generalizing}
D.~Liebhold, G.~Nebe, and M.~{\'A}. V{\'a}zquez-Castro.
\newblock Generalizing subspace codes to flag codes using group actions.
\newblock In {\em Network Coding and Subspace Designs}, pages 67--89. Springer,
  2018.

\bibitem{wang2003linear}
H.~Wang, C.~Xing, and R.~Safavi-Naini.
\newblock \highlight{Linear authentication codes: bounds and constructions}.
\newblock {\em IEEE Transactions on Information Theory}, 49(4):866--872, 2003.

\bibitem{xia2009johnson}
S.-T. Xia and F.-W. Fu.
\newblock \highlight{Johnson type bounds on constant dimension codes}.
\newblock {\em Designs, Codes and Cryptography}, 50(2):163--172, 2009.

\end{thebibliography}

\end{document}